\documentclass[reqno, numberwithin]{amsart}
\usepackage{amsmath,amssymb,latexsym, amsfonts, amscd, amsthm,color}
\usepackage{verbatim} 

\usepackage{tikz-cd}

\usepackage{xcolor}
\usepackage{hyperref}

\usepackage{multirow,longtable,changepage,tabularx,float}

\hypersetup{colorlinks=false,linkbordercolor=red,linkcolor=green,pdfborderstyle={/S/U/W 1}}

\usepackage{bbm}
\usepackage{enumerate}
\usepackage[all]{xy}

\allowdisplaybreaks
\numberwithin{equation}{section}

\newtheorem{thm}{Theorem}[section]
\newtheorem{pro}[thm]{Proposition}

\newtheorem{cor}[thm]{Corollary}
\newtheorem{conj}[thm]{Conjecture}

\theoremstyle{remark}

\newtheorem{rem}[thm]{Remark}

\theoremstyle{definition}

\definecolor{cerulean}{rgb}{0,.48,.65} 
\definecolor{magenta}{rgb}{.5,0,.5} 
\definecolor{dred}{rgb}{.5,0,0} 
\definecolor{green}{rgb}{0,.5,0} 
\definecolor{blue}{rgb}{0,0,0.7} 
\definecolor{black}{rgb}{0,0,0} 
\definecolor{dgreen}{rgb}{0,.3,0} 
\definecolor{vdred}{rgb}{.3,0,0} 
\definecolor{red}{rgb}{1,0,0} 
\definecolor{salmon}{rgb}{0.98,0.50,0.45} 
\definecolor{gray}{rgb}{.5,.5,.5} 
\definecolor{seagreen}{rgb}{0.13,0.70,0.67} 
\definecolor{chartreuse}{rgb}{0.40,0.80,0.00}
\definecolor{cornflower}{rgb}{0.39,0.58,0.93} 
\definecolor{gold}{rgb}{0.80,0.68,0.00}

\DeclareMathOperator*{\Irr}{Irr}

\DeclareMathOperator*{\Gal}{Gal}

\DeclareMathOperator*{\Sp}{Sp}

\DeclareMathOperator*{\U}{U}
\DeclareMathOperator*{\SO}{SO}

\newcommand{\SL}{\mathrm{SL}} 	
\newcommand{\GL}{\mathrm{GL}} 
\newcommand{\GO}{\mathrm{GO}} 

\newcommand{\GSO}{\mathrm{GSO}} 	
\newcommand{\GSpin}{\mathrm{GSpin}} 	
\DeclareMathOperator*{\GSp}{GSp}

\newcommand{\ad}{\mathrm{ad}} 	
\newcommand{\Ad}{\mathrm{Ad}} 	

\DeclareMathOperator*{\der}{der}

\DeclareMathOperator*{\Res}{Res}

\DeclareMathOperator*{\Hom}{Hom}

\newcommand{\Sym}{\operatorname{Sym}}

\newcommand{\vp}{\varphi}
\newcommand{\tvp}{\widetilde{\varphi}}
\newcommand{\ttau}{\widetilde{\tau}}
\newcommand{\tchi}{\widetilde{\chi}}
\newcommand{\teta}{\widetilde{\eta}}

\newcommand{\St}{\mathsf{St}}
\newcommand{\tsigma}{\widetilde{\sigma}}

\newcommand{\ts}{\widetilde{\sigma}}

\newcommand{\s}{\cong}			 
\newcommand{\pp}{\mathcal{P}} 	
\newcommand{\OO}{\mathcal{O}}

\newcommand{\CC}{\mathbb{C}}

\newcommand{\QQ}{\mathbb{Q}}
\newcommand{\ZZ}{\mathbb{Z}}

\def\bM{\mathbf M}

\def\L{\mathcal L}

\newcommand{\tB}{\widetilde{B}}
\newcommand{\wG}{\widehat{G}}
\newcommand{\tG}{\widetilde{G}}

\newcommand{\tM}{\widetilde{M}}
\newcommand{\tbM}{\widetilde{\bold M}}

\newcommand{\tT}{\widetilde{T}}
\newcommand{\tU}{\widetilde{U}}

\title[Representations and Adjoint $L$-Function for $\GSpin_4$ and $\GSpin_6$]
{Representations of the $p$-adic $\GSpin_4$ and $\GSpin_6$ and the Adjoint $L$-Function}

\author[Mahdi Asgari and Kwangho Choiy]{Mahdi Asgari and Kwangho Choiy}

\address{Mahdi Asgari\\
Department of Mathematics,
Oklahoma State University,
Stillwater, OK 74078-1058,
U.S.A.}
\email{asgari@math.okstate.edu}

\address{Kwangho Choiy\\
School of Mathematical and Statistical Sciences,
Southern Illinois University,
Carbondale, IL 62901-4408,
U.S.A.}
\email{kchoiy@siu.edu}

\begin{document}

\begin{abstract} 
We prove a conjecture of B. Gross and D. Prasad about determination of generic $L$-packets in terms of the analytic properties of the adjoint $L$-function for $p$-adic general even spin groups of semi-simple ranks 2 and 3.  We also explicitly write the adjoint $L$-function for each $L$-packet in terms of the local Langlands $L$-functions for the general linear groups.  
\end{abstract}
\maketitle 

\section{Introduction} \label{intro}
In this article, we provide further details on the local $L$-packets for the non-Archimedean split general spin groups $\GSpin_4$ and $\GSpin_6$, 
following our earlier work \cite{acgspin}.  We then use our explicit description of these $L$-packets to prove a conjecture of B. Gross and D. Prasad \cite{gross22,grossprasad92}
determining which of the $L$-packets are ``generic'' (i.e., contain an irreducible representation with a Whittaker model) 
in terms of the analytic properties at $s=1$ of the adjoint $L$-function of the packet.  
We also write the adjoint $L$-function for each $L$-packet in terms of the local Langlands $L$-functions of the general linear groups. 
In addition to details about the representations that our results provide,  
given that the adjoint $L$-functions have a significant role in the Gan-Gross-Prasad conjectures, we expect that our results in this paper 
would be helpful in that direction as well. Particularly striking is the generalization of the Gan-Gross-Prasad to the non-tempered case \cite{ggp20} where 
the relevant adjoint $L$-function does have a pole at $s=1$.

Let $F$ be a $p$-adic field of characteristic zero.  Denote by $W_F$ the Weil group of $F$ and let $W'_F = W_F \times \SL_2(\CC)$ be the Weil-Deligne group of $F$. 
Let $G$ be a connected, reductive, linear algebraic group over $F$.  The local Langlands Conjecture (LLC) predicts a surjective, finite-to-one 
map $\L$ from the set $\operatorname{Irr}(G)$ of equivalence classes of irreducible, smooth, complex representations of $G(F)$ 
to the set $\Phi(G)$ of $\wG$-conjugacy classes of $L$-parameters of $G(F)$, i.e., admissible homomorphisms $\phi : W'_F \longrightarrow {}^LG$. 
Here,  ${}^LG$ denotes the $L$-group of $G$ with $\wG = {}^LG^0$ its connected component, i.e., the complex dual of $G$ \cite{bo79}. 
Among other properties, the map $\L$ is supposed to preserve the local $L$-, $\epsilon$-, and $\gamma$-factors. 
Moreover, the (finite) fibers $\Pi_\phi$, for $\phi \in \Phi(G)$, of the map $\L$ are called the $L$-packets of $G$ and 
their structures are expected to be controlled by certain finite subgroups of $\wG$.

Consider the split general spin groups $G=\GSpin_4$ and $G=\GSpin_6$, of type $D_2 = A_1 \times A_2$ and $D_3 = A_3$ respectively, 
whose algebraic structure we review in Section \ref{sec:gspin46-structure}.  We constructed most of the $L$-packets for these two groups in \cite{acgspin} 
and proved that they satisfy the expected properties of preservation of the local factors and their internal structure.  We review and complete 
the construction of these $L$-packets. In particular, using the classification of representations of $GL_n,$ we give more explicit descriptions 
of the $L$-packets for $\GSpin_4$ and $\GSpin_6$ in terms of given representations of $\GL_2 \times \GL_2$ and $\GL_4 \times GL_1,$ respectively. 
As a byproduct, we are able to give the criteria for determining the size of the $L$-packets for $\GSpin_4$ and $\GSpin_6$ 
(see Sections \ref{genericclassification4} and \ref{genericclassification6}).

The known cases of the LLC for the $p$-adic groups include $\GL_n$\cite{ht01, he00, scholze13};
$\SL_n$ \cite{gk82}; 
 non-quasi-split $F$-inner forms of $\GL_n$ and $\SL_n$ \cite{hs11, abps13};   
 $\GSp_4$ and $\Sp_4$ \cite{gt, gtsp10};
 non-quasi-split $F$-inner form $\GSp_{1,1}$ of $\GSp_4$ \cite{gtan12};
 $\Sp_{2n}, \SO_{n},$ and quasi-split $\SO^*_{2n}$ \cite{art12};
 $\U_n$  \cite{rog90, mok13};
non quasi-split $F$-inner forms of $\U_n$ \cite{rog90, kmsw14};
non-quasi-split $F$-inner form $\Sp_{1,1}$ of $\Sp_4$ \cite{ch15};
$\GSpin_{4}, \GSpin_{6}$ and their inner forms \cite{acgspin};
$\GSp_{2n}$ and $\GO_{2n}$ \cite{xumathann}.

Going back to the case of general $G$, assume that $\rho$ is a finite-dimensional complex representation of ${}^LG$.  When LLC is known, 
one can define the local Langlands $L$-functions 
\[ 
L(s, \pi, \rho) = L(s, \rho \circ \phi) 
\] 
for each $\pi \in \Pi_\phi$. 
Here, the $L$-factors on the right hand side are the Artin local factors associated to the given representation of $W'_F$.

B. Gross and D. Prasad conjectured (in the generality of quasi-split groups) that the local 
$L$-packet $\Pi_\phi(G)$ is generic if and only if the adjoint $L$-function $L(s, \Ad \circ \phi)$ is regular at $s=1$ 
\cite[Conj. 2.6]{grossprasad92}.  
Here, $\Ad$ denotes the adjoint representation of ${}^LG$ on the dual Lie algebra $\widehat{\mathfrak g}$ of $\wG$. 
(Note that in the body of this paper we use $\Ad$ exclusively for the restriction of the adjoint representation to the 
derived group of $\widehat{\mathfrak g}$ to distinguish it from the full adjoint $L$-function, which would have an 
extra factor of the $L$-function for the trivial character when $\widehat{\mathfrak g}$ has a one-dimensional center.)

We prove the above conjecture for the groups $\GSpin_4$ and $\GSpin_6$ as a consequence of our construction of 
the $L$-packets for these groups.  In fact, we prove the conjecture for a larger class of groups $G = G_{m,n}^{r,s}$, 
which are given as subgroups of $\GL_m \times \GL_n$ satisfying a certain determinant equality \eqref{Gmn}.  
We are able to work in the slightly larger generality because, as in the construction of the $L$-packets, we use the 
approach of restricting representations from $\GL_m(F) \times \GL_n(F)$ to the subgroup $G$.

Moreover, we also give the adjoint $L$-function in all cases explicitly in terms of local Langlands $L$-functions 
of the general linear groups.  While we are able to prove the Gross-Prasad conjecture already without the explicit 
knowledge of the adjoint $L$-function, the explicit description of the adjoint $L$-function certainly also verifies the conjecture 
and we include it here since it may lead to other number theoretic or representation theoretic results.

Finally, we take this opportunity to correct a few inaccuracies in \cite{acgspin}. They do not affect the main 
results in that paper and fix some errors in our description of the $L$-packets. The details are given in 
Section \ref{sec:corrections}.

\subsection*{Acknowledgements}
We are grateful to Behrang Noohi and 
Ralf Schmidt for helpful discussions. 
We also thank B. Gross for his interest in this paper and clarifying the history of his conjecture and the context in 
which it was made. 

K. Choiy was supported by a gift from the Simons Foundation (\#840755). 

\section{Preliminaries} 

\subsection{Local Langlands Correspondence (LLC)} 
Let $p$ be a prime number and let $F$ be a $p$-adic field of characteristic zero, 
i.e., a finite extension of $\QQ_p$. We fix an algebraic closure $\bar{F}$ of $F.$ 
Denote the 
ring of integers of $F$ by $\OO_F$ and its unique maximal ideal by $\pp_F$. 
Moreover, let $q$ denote the cardinality of the residue field $\OO_F / \pp_F$ 
and fix a uniformizer $\varpi$ with $| \varpi |_F = q^{-1}$. Also, let $W_F$ denote 
the Weil group of $F$, $W'_F$ the Weil-Deligne group of $F$, and 
$\Gamma$ the absolute Galois group $\Gal(\bar{F} / F)$.  
Throughout the paper, we will use the notation $\nu(\cdot) = |\cdot|_F$.

Let $G$ be a connected, reductive, linear algebraic group over $F$.  
Fixing $\Gamma$-invariant splitting data we define the $L$-group of $G$ as 
a semi-direct product ${}^LG := \widehat{G} \rtimes \Gamma$, 
where $\widehat{G} = {}^LG^0$ denotes the connected component 
of the $L$-group of $G,$ i.e., the complex dual of $G$ (see \cite[\S 2]{bo79}).

LLC (still conjectural in this generality) asserts that 
there is a surjective, finite-to-one map from the set $\Irr(G)$ of isomorphism classes of 
irreducible smooth complex representations of $G(F)$ to the set $\Phi(G)$ of 
$\widehat{G}$-conjugacy classes 
of $L$-parameters, i.e., admissible homomorphisms 
$\vp: W'_F \longrightarrow {}^LG$.  

Given $\vp \in \Phi(G),$ its fiber $\Pi_{\vp}(G)$, which is called an $L$-packet for $G,$ 
is expected to be controlled by a certain finite group living in the complex dual group $\widehat{G}.$
Furthermore, for $\pi \in \Pi_{\vp}(G)$ and $\rho$ a finite dimensional algebraic representation 
of ${}^LG$ one defines the local factors 
\begin{eqnarray} 
L(s, \pi, \rho) &=& L(s, \rho \circ \phi), \\
\epsilon(s,\pi, \rho,\psi) &=& \epsilon(s, \rho \circ \phi, \psi), \\
\gamma(s,\pi, \rho,\psi) &=& \gamma(s, \rho \circ \phi, \psi). 
\end{eqnarray} 
provided that LLC is known for the case in question. Here, the factors on the right are Artin factors.

\subsection{The Adjoint $L$-Function} 
What we recall in this subsection holds for $G$ quasi-split (\cite[\S 2]{grossprasad92}). However, for simplicity 
we will take $G$ to be split over $F$ since the groups we are working with in this article are split. 
When $G$ is split over $F$, we may replace the $L$-group ${}^LG$ by its 
connected component $\wG = {}^LG^0$. Take $\rho$ to be the adjoint action of $\wG$ on its Lie algebra. 
Then we obtain the adjoint $L$-function $L(s, \pi, \Ad_{\wG}) = L(s, \Ad_{\wG} \circ \phi)$ for 
all $\pi \in \Pi_{\vp}(G)$. 
The following is a conjecture of D. Gross and D. Prasad (see \cite[Conj. 2.6]{grossprasad92}).  

\begin{conj} \label{GP-R conj}
$\Pi_{\vp}(G)$ contains a generic member if and only if $L(s, \Ad_{\wG} \circ \phi)$ is regular at $s=1$. 
(Equivalently, $\pi$ is generic if and only if $L(s,\pi,\Ad_{\wG})$ is regular at $s=1$.) 
\end{conj} 

The conjecture is known in many cases in which the LLC is known. To mention a few, it was verified for $\GL_n$ 
by B. Gross and D. Prasad \cite{grossprasad92}, for $\GSp_4$ in \cite{gt} and, for 
non-supercuspidals, in \cite{mahdi-ralf}, and for $\SO$ and $\Sp$ groups, it follows from the work 
of Arthur on endoscopic classification \cite{art12}. 
We will verify this conjecture for the small rank split groups $\GSpin_4$ and $\GSpin_6$.

\subsection{The Groups $\GSpin_4$ and $\GSpin_6$} \label{sec:gspin46-structure} 
We gave detailed information about the structure of these two groups (as well as their inner forms) 
in \cite[\S 2.2]{acgspin}. 
For now we just recall the incidental isomorphisms 
\begin{eqnarray} 
{\GSpin}_4  &\cong& \left\{ (g_1, g_2) \in {\GL}_2 \times {\GL}_2 : \det g_1 = \det g_2 \right\} \\
{\GSpin}_6  &\cong& \left\{ (g_1, g_2) \in {\GL}_1 \times {\GL}_4 : g_1^2 = \det g_2 \right\}.
\end{eqnarray}

While our main interests in this article are the split general spin groups $\GSpin_4$ 
and $\GSpin_6$, for the purposes of 
Conjecture \ref{GP-R conj} it is no more difficult, and perhaps also more natural, 
to consider a slightly more general setup as follows.

Fix integers $m,n \ge 1$ and $r,s \ge 1$ and assume that $\operatorname{gcd}(r,s)=1$.  
Define 
\begin{equation} \label{Gmn} 
G = G_{m,n}^{r,s} := \left\{ (g,h) \in \GL_m \times \GL_n \mid 
(\det g)^r = (\det h)^s \right\} 
\end{equation}

\begin{pro}
The group $G_{m,n}^{r,s}$ is a split, connected, reductive, linear algebraic group over $F$.  
\end{pro} 

\begin{proof} 
Let $X = (X_{ij})$ and $Y = (Y_{kl})$ be $m \times m$ and $n \times n$ matrices, respectively. 
It is clear that $G_{m,n}^{r,s}$, being an almost direct product of $SL_m \times \SL_n$ and a 
torus, is reductive. 
The only issue that requires justification is that the polynomial 
$ f(X,Y) = (\det X)^r - (\det Y)^s$ is irreducible in $F[X_{ij}, Y_{kl}]$ if and only if $d = \gcd(r,s) = 1$.  
It is clear that if $d > 1$, then $f$ is reducible since it would be divisible by $(\det X)^{(r/d)} - (\det Y)^{(s/d)}$.  
It remains to show that if $d=1$, then $f(X,Y)$ is irreducible.  This assertion should be easy to see 
via elementary arguments considering the polynomials in a possible factorization of $f$.  However, we 
prove it below as a special case of a more general fact.

Assume that $f(x,y)$ is an (arbitrary) irreducible polynomial in $F[x,y]$. Let 
\[ p(x_1,x_2,\dots,x_a) \in F[x_1,x_2,\dots,x_a] \quad \mbox{ and } \quad 
p(y_1,y_2,\dots,y_b) \in F[y_1,y_2,\dots,y_b] \] 
be two polynomials such that 
$p-\alpha$ and $q-\alpha$ are irreducible for all constants $\alpha$. 
Then, $f(p,q)$ is irreducible in $F[x_1,x_2,\dots,x_a,y_1,y_2,\dots,y_b]$.

Our Proposition would clearly follow from the above assertion since $(\det - \alpha)$ is always 
an irreducible polynomial and it is well-known that the two-variable polynomial $x^r - y^s$ is irreducible 
in $F[x,y]$ provided that $d = \gcd(r,s) = 1$.

To prove the assertion above, we proceed as follows. By base extension to an algebraic closure 
we may assume, without loss of generality, that $F$ is algebraically closed.

Let $A$ be the subscheme of $\operatorname{Spec} F[x_1,x_2,\dots,x_a,y_1,y_2,\dots,y_b]$ 
defined by $f(p,q)$, and 
let $B$ be the subscheme of $\operatorname{Spec} F[x,y]$ defined by $x^r-y^s$.  
The latter is irreducible since $x^r-y^s$ is an irreducible polynomial by our assumption that $d=1$.  
There is a natural map $A \to B$ which has irreducible (geometric) fibers. 
The result now follows from the following claim.

Claim: Let $g : A \to B$ be an open morphism of schemes of finite type over 
an algebraically closed field $F$ such that the (geometric) fibers of $g$ are irreducible 
and $B$ is irreducible. Then $A$ is irreducible.

To see the claim let $U$ be an open in $A$.  We want to show that for any other open $V$, 
we have that $U \cap V$ is nonempty. Since $B$ is irreducible and $g$ is open, we have that 
$g(U) \cap g(V)$ is nonempty so there is a fiber $F_0$ of $g$ such that 
$F_0 \cap U$ and $F_0 \cap V$ are nonempty.  Hence, by irreducibility of $F_0$, 
they have a nonempty intersection in $F_0$.  In particular, $U \cap V$ is nonempty, 
which gives the claim.

It only remains to check that the map $A \to B$ above is open. 
In fact, it is flat since it is a base extension of the cartesian product of two flat morphisms
$p : \operatorname{Spec} F[x_1,...,x_a] \to \operatorname{Spec} F[x]$ and 
$q : \operatorname{Spec} F[y_1,...,y_b] \to \operatorname{Spec} F[y]$.  
(Here, we are using the fact that $\operatorname{Spec} F[x]$ is a curve.) This finishes the proof. 
\end{proof}

Of particular interest to us in this paper are the cases 
\begin{itemize} 
\item $m=n=2$ and $r=s=1$, when $G = \GSpin_4$, and 
\item $m=1$, $n=4$ and $r=2$, $s=1$, when $G = \GSpin_6$. 
\end{itemize}

The (connected) $L$-group of $G$ is 
\begin{equation} \label{L-gp for Gmn}
{{}^L {G}_{m,n}^{{r,s} \,0}} = \widehat{G} \s  ({\GL}_m(\CC) \times {\GL}_n(\CC))  / \{ (z^{-r}I_m, z^sI_n) : z \in \CC^{\times}   \} 
\end{equation}
and we have the exact sequence 
\begin{equation} \label{ex seq L-gp for Gmn}
1 \longrightarrow  \{ (z^{-r}I_m, z^sI_n) : z \in \CC^{\times}   \} \s \CC^{\times} 
\longrightarrow {\GL}_m(\CC) \times {\GL}_n(\CC) 
\xrightarrow{pr_{m,n}^{r,s}}
\widehat{G_{m,n}^{r,s}} \longrightarrow 1. 
\end{equation}

\subsection{Computation of the Adjoint $L$-Function for $G$}

Let $\pi$ be an irreducible admissible representation of $G(F)$.  There 
exist irreducible admissible representations $\pi_m$ and $\pi_n$ of $\GL_m(F)$ and 
$\GL_n(F)$, respectively, such that 
\begin{equation} \label{pi-mn} 
\pi \hookrightarrow {\Res}_{G(F)}^{\GL_m(F) \times \GL_n(F)} \left( \pi_m \otimes \pi_n \right). 
\end{equation}  
Let $\Ad_{\wG}$ denote the adjoint action of $\wG$ on its Lie algebra 
\begin{equation}  \label{gmn}
\widehat{\mathfrak g} = \left\{ (X,Y) \in \mathfrak{gl}_m(\CC) \times \mathfrak{gl}_n(\CC) \mid 
r \operatorname{tr}(X) = s \operatorname{tr}(Y) \right\}. 
\end{equation}  
In what follows, let us write 
\begin{equation}
\Ad_{\wG} = \operatorname{triv} \oplus \Ad 
\end{equation} 
and for $i \in \{m,n\}$ we similarly write 
$\Ad_i = \Ad_{\widehat{GL}_i} = \operatorname{triv} \oplus \Ad,$
where $\Ad$ here denotes the action of $\GL_i(\CC)$ on the space 
of traceless $i \times i$ complex matrices ${\mathfrak sl}_i(\CC)$.

Let $\phi_\pi : W_F \times \SL_2(\CC) \to \wG$ be the $L$-parameter of $\pi$ and let 
$\phi_i  : W_F \times \SL_2(\CC) \to \GL_i(\CC)$, $i=m,n$, be the $L$-parameter of $\pi_i$. 
Recall by \eqref{ex seq L-gp for Gmn} that we have a natural map 
\begin{equation}  
pr = {pr_{m,n}^{r,s}} : \GL_m(\CC) \times \GL_n(\CC) \longrightarrow \wG.  
\end{equation}  
Then we have 
\begin{equation}  
\phi_\pi = pr \circ (\phi_m \otimes \phi_n). 
\end{equation}  
Since the subgroup 
$\{ (z^{-r}I_m, z^sI_n) : z \in \CC^{\times}   \}$ 
is central in $\GL_m(\CC) \times \GL_n(\CC)$ 
the following diagram commutes. 
 
\[ \begin{tikzcd}[row sep=huge, column sep=huge] 
& \GL_m(\CC) \times \GL_n(\CC) \arrow[r, "\Ad_m \otimes \Ad_n"] \arrow[dd, "pr"]
& \operatorname{Aut}_\CC\left( \mathfrak{gl}_m(\CC) \times \mathfrak{gl}_n(\CC) \right) 
\arrow[dd] \\ 
W_F \times \SL_2(\CC) \arrow[ur, "\phi_m \otimes \phi_n"] \arrow[dr, "\phi_\pi"] 
& &  \\ 
& \wG \arrow[r, "\Ad_{\wG}"] 
& \operatorname{Aut}_\CC\left(\widehat{\mathfrak g} \right)
\end{tikzcd} \] 

Note that the adjoint action $\Ad_m$ of $\GL_m(\CC)$ on 
${\mathfrak gl}_m(\CC)$ 
preserves the trace, and similarly for $n$, so we obtain a right downward arrow 
by simply restricting any automorphism to the set of those pairs satisfying the trace 
equality in \eqref{gmn}. 
We have  
\begin{eqnarray}  \label{Ad-L}
\nonumber L(s, 1_{F^\times}) L(s, \pi, \Ad) \cdot L(s, 1_{F^\times}) &=& L(s, \pi, \Ad_{\wG}) \cdot L(s, 1_{F^\times}) \\
\nonumber &=& L(s, \Ad_{\wG} \circ \phi_\pi) \cdot L(s, 1_{F^\times}) \\
\nonumber &=& L\left(s, (\Ad_m \otimes \Ad_n) \circ (\phi_m \otimes \phi_n) \right) \\
\nonumber &=& L(s, \Ad_m \circ \phi_m) L(s, \Ad_n \circ \phi_n) \\
\nonumber &=& L(s, \pi_m, \Ad_m) L(s, \pi_n, \Ad_n) \\ 
&=& L(s, 1_{F^\times})^2 L(s, \pi_m, \Ad) L(s, \pi_n, \Ad). 
\end{eqnarray}  
Therefore, we obtain the more convenient equality 
\begin{equation}\label{Ad-L2} 
L(s, \pi, \Ad) = L(s, \pi_m, \Ad) L(s, \pi_n, \Ad), 
\end{equation} 
which holds thanks to our choice of the notation $\Ad$. 
In Section \ref{pf-conj} this relation helps verify Conjecture \ref{GP-R conj} for the groups 
of interest to us.

\section{Genericity and The Conjecture of B. Gross and D. Prasad} \label{conj-sec}

\subsection{Restriction of Generic Representations} \label{restrictionofgeneric}
Let us write $\square^D$ for the group $\Hom(\square, \CC^{\times})$ of all continuous characters on 
a topological group $\square$. Dente by $\square_{\der}$ the derived group of $\square.$
Let $G$ and $\tG$ be connected, reductive, linear, algebraic groups over $F$ satisfying the property that
\begin{equation} \label{cond on G}
G_{\der} = \tG_{\der} \subseteq G \subseteq \tG.
\end{equation}
For any connected, reductive, linear, algebraic group $\square$ over $F,$ we write $\Irr_{\rm sc}(\square)$ 
and $\Irr_{\rm esq}(\square)$ for the set of equivalence classes of supercuspidal and essentially square-integrable 
representations of $\square(F),$ respectively.

Assume $\tG$ and $G$ to be $F$-split.
Let $\tB$ be a Borel subgroup of $\tG$ with Levi decomposition $\tB=\tT \tU.$ Then $B = \tB \cap G$ is a 
Borel subgroup of $G$ with $B=T U$.  Note that $T = \tT \cap G$ and $\tU=U.$
Let $\psi$ be a generic character of $U(F)$.  From \cite[Proposition 2.8]{tad92} we know that 
given a $\psi$-generic irreducible representation $\ts$ of $\tG(F)$ 
we have a unique $\psi$-generic $\sigma$ of $G(F)$ such that 
\[
\sigma \hookrightarrow {\Res}^{\tG}_{G} (\ts).
\] 
The generic character associated with $\sigma$ is not unique though.

\begin{pro} \label{conjugation}
Each generic character associated with $\sigma$ is determined up to the action of $\tT(F)/T(F).$ 
\end{pro}

\begin{proof}
We let $\ts \in \Irr(\tG)$ be $\psi$-generic. Then there is a unique $\psi$-generic $\sigma_\psi \in \Pi_{\ts}(G)$.  
On the other hand, for each $\sigma \in \Pi_{\ts}(G)$ there exists $t \in \tT(F)/T(F) \s \tG/G(F)$ such that 
$\sigma={^t}\sigma_\psi,$ where $^t\sigma_\psi(g)=\sigma(t^{-1}gt).$ This implies that $\sigma$ is $^t\psi$-generic. 
Here $^t\psi$ is defined as $^t\psi(u)=\psi(t^{-1}ut).$
\end{proof}

\begin{rem} 
We say $\sigma \in \Irr(G)$, resp. $\tsigma \in \Irr(\tG)$, is generic if it is $\psi$-generic with respect to some generic character $\psi$. 
With this notation, $\sigma \in \Irr(G)$ is generic if and only if is $\tsigma \in \Irr(\tG).$
\end{rem}

\subsection{
Criterion for Genericity
} \label{pf-conj}
In this section we verify Conjecture \ref{GP-R conj} for the small rank general spin groups we are considering in this article.

\begin{thm} 
Let $G = G_{m,n}^{r,s}$ be the group defined in \eqref{Gmn}.  Let $\pi$ be an irreducible admissible representation of $G(F)$.  
Then $\pi$ is generic if and only if $L(s, \pi, \Ad)$ is regular at $s=1$. 
\end{thm} 

\begin{proof}  
Given $\pi$ there exist irreducible admissible representations 
$\pi_m$ of $\GL_m(F)$ and $\pi_n$ of $\GL_n(F)$ such that $\pi$ 
is a subrepresentation of the restriction to $G(F)$ of $\pi_m \otimes \pi_n$ 
as in \eqref{pi-mn}.  Now, $\pi$ is generic if and only if both $\pi_m$ and 
$\pi_n$ are generic. By the truth of Conjecture \ref{GP-R conj} for the 
general linear groups, the latter is equivalent to both 
$L(s, \pi_m, \Ad)$ and $L(s, \pi_n, \Ad)$ being regular at $s=1$. 
Hence, by \eqref{Ad-L2} and the fact that neither of the $L$-functions can have 
a zero at $s=1$, we have that 
$\pi$ is generic if and only if $L(s, \pi, \Ad)$ is regular at $s=1$. 
This proves the theorem. 
\end{proof} 

As we observed in Section \ref{sec:gspin46-structure}, the split groups $\GSpin_4$ and $\GSpin_6$ are special cases 
of $G_{m,n}^{r,s}$. Therefore, we have the following.

\begin{cor} 
Conjecture \ref{GP-R conj} holds for the groups $\GSpin_4$ and $\GSpin_6$. 
\end{cor} 

\section{Representations of $\GSpin_4$} \label{genericclassification4}
In this section we list all the irreducible representations of $\GSpin_4(F)$ and then calculate 
their associated adjoint $L$-function explicitly.  To this end, we give the nilpotent 
matrix associated to their parameter in each case.

\subsection{The Reprsentations}
\subsubsection{Classification of representations of $\GSpin_4$} 

Following \cite{acgspin}, we have
\begin{equation} 
1 \longrightarrow {\GSpin}_4(F) \longrightarrow {\GL}_2(F) \times {\GL}_2(F) \longrightarrow F^{\times} \longrightarrow 1.  
\end{equation} 
Recall that 
\begin{equation} \label{convenient description of GSpin4(F)}
{\GSpin}_4(F) \s \{(g_1, g_2) \in {\GL}_2(F) \times {\GL}_2(F) : \det g_1 = \det g_2 \} , 
\end{equation}
\begin{equation} \label{L-gp for GSpin4}
{^L{\GSpin}_4} = \widehat{{\GSpin}_4}  = {\GSO}_{4}(\CC) \s  ({\GL}_2(\CC) \times {\GL}_2(\CC))  / \{ (z^{-1}, z) : z \in \CC^{\times}   \},  
\end{equation}
and
\begin{equation} \label{ex seq L-gp for GSpin4}
1 \longrightarrow \CC^{\times} \longrightarrow {\GL}_2(\CC) \times {\GL}_2(\CC) 
\overset{pr_4}{\longrightarrow} \widehat{{\GSpin}_4} \longrightarrow 1.  
\end{equation} 
When convenient, we view $\GSO_4$ as the group similitude orthogonal $4 \times 4$ matrices 
with respect to the anti-diagonal matrix 
\begin{equation} \label{J-matrix} 
J = J_4 = 
\begin{bmatrix} 
0&0&0&1 \\ 
0&0&1&0 \\
0&1&0&0 \\
1&0&0&0 
\end{bmatrix}. 
\end{equation} 
The Lie algebra of this group is also defined with respect to $J$ and an element $X$ 
in this Lie algebra satisfies 
\[ {}^tX J + J X = 0.  \]

\subsubsection{Construction of the $L$-packets of $\GSpin_4$ (recalled from \cite{acgspin})}
Given $\sigma \in \Irr(\GSpin_4)$ we have a lift $\ts \in \Irr({\GL}_2 \times {\GL}_2)$ such that 
\[
\sigma \hookrightarrow {\Res}_{{\GSpin}_{4}}^{\GL_2 \times \GL_2}(\ts).
\]
It follows form the LLC for $GL_n$ \cite{ht01, he00, scholze13} that there is a unique $\tvp_{\ts} \in \Phi({\GL}_2 \times {\GL}_2)$ corresponding to the representation $\ts.$
We now have a surjective, finite-to-one map 
\begin{eqnarray} \label{L map for GSpin4}
{\L}_{4} : {\Irr}({\GSpin}_{4}) & \longrightarrow & \Phi({\GSpin}_{4}) \\
\sigma & \longmapsto & pr_4 \circ \tvp_{\ts}, \nonumber
\end{eqnarray} 
which does not depend on the choice of the lifting $\ts.$
Then, for each $\vp \in \Phi(\GSpin_4),$ all inequivalent irreducible constituents of $\ts$ constitutes the $L$-packet 
\begin{equation} \label{def of L-packet for GSpin4}
\Pi_{\vp}({\GSpin}_{4}):=\Pi_{\ts}({\GSpin}_4) = 
\left\{ \sigma \, \middle| \, \sigma \hookrightarrow {\Res}_{{\GSpin}_{4}}^{\GL_2 \times \GL_2}(\ts) \right\} \Big\slash \s.
\end{equation}
Here, $\ts$ is the member in the singleton $\Pi_{\tvp}(\GL_2 \times \GL_2)$ and 
$\tvp \in \Phi(\GL_2 \times \GL_2)$ is such that $pr_{4} \circ \tvp=\vp.$ 
We note that the construction does not depends on the choice of $\tvp,$
due to the LLC for $\GL_2$, \cite[Lemma 2.4]{gk82}, \cite[Corollary 2.5]{tad92}, and \cite[Lemma 2.2]{hs11}.  
Further details can be found in \cite[Section 5.1]{acgspin}.

\subsubsection{The $L$-parameters of ${\GL}_2$} \label{lparametersgl2}
We recall the generic representations of $\GL_2(F)$ in this paragraph. We refer to \cite{wedhorn08, ku94, gr10} for details. 
Let $\chi: F^\times \rightarrow \CC^\times$ denote a continuous quasi-character of $F^\times$.
By Zelevinski (\cite[Theorem 9.7]{zel80} or \cite[Theorem 2.3.1]{ku94}) we know that 
the generic representations of ${\GL}_2$ are: 
the supercuspidals, $\St \otimes (\chi\circ \det)$ where $\St$ denotes the Steinberg representation, 
and normally induced representations $i_{\GL_1 \times \GL_1}^{\GL_2}(\chi_1 \otimes \chi_2)$ with $\chi_1 \neq \chi_2 \nu^{\pm1}.$ 
The only non-generic representation is $\chi \circ \det.$

\subsection{Generic Representations of $\GSpin_4$}

Following \cite[Section 5.3]{acgspin}, given $\vp \in \Phi(\GSpin_4),$ fix the lift 
\[ 
\tvp=\tvp_1 \otimes \tvp_2 \in \Phi({\GL}_2\times {\GL}_2)  
\] 
with $\tvp_i \in \Phi({\GL}_2)$ such that $\vp = pr_4 \circ \tvp$. 
Let 
\[ 
\ts = \ts_1 \boxtimes \ts_2 \in \Pi_{\tvp}({\GL}_2\times {\GL}_2) 
\] 
be the unique member such that $\{\ts_i\} = \Pi_{\tvp_i}({\GL}_2).$

Recall the notation 
\[
I^{\GSpin_4}(\ts) := \left\{ \chi \in \left({\GL}_2(F) \times {\GL}_2(F) /{\GSpin}_4(F)\right)^D \, \middle| \, \ts \otimes \chi \s \ts  \right\}.
\]
Then we have
\begin{equation} \label{1to1}
\Pi_{\vp}({\GSpin}_4) \, \overset{1-1}{\longleftrightarrow} \,I^{\GSpin_4}(\ts),
\end{equation}
and we recall that, by \cite[Proposition 5.7]{acgspin}, we have 
\begin{equation} \label{stab}
I^{\GSpin_4}(\ts) = \left\{ \begin{array}{lll}
         I^{{\SL}_2}(\ts_1), & 	\mbox{if $\ts_2 \s \ts_1\teta$ for some $\teta \in (F^{\times})^D$}; \\
        I^{{\SL}_2}(\ts_1) \cap I^{{\SL}_2}(\ts_2), & \mbox{if $\ts_2 \not\s \ts_1\teta$ for any $\teta \in (F^{\times})^D$}.
				\end{array} \right.
\end{equation}

\subsubsection{Irreducible Parameters} \label{irr parameter gspin4}

Let $\vp \in \Phi(\GSpin_4)$ be irreducible.  Then $\tvp,$ $\tvp_1,$ and $\tvp_2$ are all irreducible. 
By Section \ref{restrictionofgeneric}, we have the following.
\begin{pro} 
Let $\vp \in \Phi(\GSpin_4)$ be irreducible.  Then every member in $\Pi_{\vp}(\GSpin_4)$ is supercuspidal and generic.
\end{pro}
To study the internal structure of $\Pi_{\vp}(\GSpin_4)$, by \eqref{1to1}, 
we need to know the structure of $I^{\GSpin_4}(\ts)$, as we now recall from \cite{acgspin}.  
\begin{enumerate}[$\mathfrak{gnr}$-(a)]
\item When $\ts_2 \s \ts_1\teta$ for some $\teta \in (F^{\times})^D,$ we have
\[ 
I^{\GSpin_4}(\ts) \s 
\left\{ 
\begin{array}{llll}
         \{1\}, & 
         \mbox{if $\tvp_1$ (and hence also $\tvp_2$) is primitive or non-trivial on $\SL_2(\CC)$}; \\
         \ZZ/2\ZZ, &  
         \mbox{if $\tvp_1$ (and hence also $\tvp_2$) is dihedral w.r.t. one quadratic extension}; \\ 
         (\ZZ/2\ZZ)^2, &  
         \mbox{if $\tvp_1$ (and hence also $\tvp_2$) is dihedral w.r.t. three quadratic extensions}.
\end{array} \right. 
\]
\item When $\ts_2 \not\s \ts_1\teta$ for any $\teta \in (F^{\times})^D$, then by \eqref{stab} we have  
\[ 
I^{\GSpin_4}(\ts) \s \{1\} ~\text{or } \ZZ/2\ZZ.
\]
Since $\ts_2 \not\s \ts_1\teta$ for any $\teta \in (F^{\times})^D$,  
the case of both 
$\tvp_1$ and $\tvp_2$ being diredral w.r.t. three quadratic extensions is excluded.  
Thus, we have the following list:
\begin{itemize}
\item If at least one of $\tvp_i$ is primitive, then $I^{\GSpin_4}(\ts) \s \{1\}.$
\item If both are dihedral, then $I^{\GSpin_4}(\ts) \s \ZZ/2\ZZ.$
\end{itemize}
\end{enumerate}
From \cite[Proposition 2.1]{acgspin}, we recall the identification 
\begin{equation} \label{indentity}
\Delta^\vee = \left\{ \beta^\vee_1 = f^*_{11} - f^*_{12}, 
\beta^\vee_2 = f^*_{21} - f^*_{22} \right\}, 
\end{equation}
using the notation $f_{ij}$ and $f^*_{ij},$ $1 \le i, j \le 2,$ 
for the usual $\ZZ$-basis of characters and cocharacters of  $\GL_2 \times \GL_2$ 
and $\beta_1, \beta_2$ denote the simple roots of $\GSpin_4$.  We can use this 
identification to relate the nilpotent matrices associated to the parameters of $\GL_2 \times \GL_2$ 
and $\GSpin_4$, respectively. 

For both (a) and (b) above, we have 
\[
N_{{\GL}_2(\CC) \times {\GL}_2(\CC)} = 
\left( \begin{bmatrix} 0&0 \\  0&0 \end{bmatrix} , 
\begin{bmatrix} 0&0 \\  0&0 \end{bmatrix} \right) 
\overset{\eqref{indentity}}{\Longleftrightarrow}
N_{{\GSO}_4(\CC)}= 
0_{4\times4}. 
\]
\begin{rem}
We note that case (b) above was mentioned, less precisely, in \cite[Remark 5.10]{acgspin}.
\end{rem}

\subsubsection{Reducible Parameters}  \label{non-irr parameter gspin4}

If $\vp \in \Phi(\GSpin_4)$ is reducible, then at least one $\tvp_i$ must be reducible. 
Since the number of irreducible constituents in $\Res_{\SL_2}^{{\GL}_2} (\ts_i)$ is at most 2, 
we have 
$I^{\SL_2}(\ts_i) \s \{1\}, ~\text{or } \ZZ/2\ZZ.$
This implies that 
\[
I^{\GSpin_4}(\ts) \s \{1\}, ~ \text{or } \ZZ/2\ZZ.
\]

If $\tvp_i$ is reducible and generic, then
$\ts_i$ is either the Steinberg representation twisted by a character or an irreducibly induced representation from the Borel subgroup of ${\GL}_2.$ We make case-by-case arguments as follows.

\begin{enumerate}[$\mathfrak{gnr}$-(i)]

\item Note that the Steinberg representation of $\GL_2 \times \GL_2$ is of the form ${\St}_{\GL_2} \boxtimes {\St}_{\GL_2}.$ We have
\begin{equation} \label{generici}
{\Res}^{\GL_2 \times \GL_2}_{\GSpin_4} ({\St}_{\GL_2} \boxtimes {\St}_{\GL_2}) = {\St}_{\GSpin_4}
\end{equation}
and
\[
{\Res}^{\GL_2 \times \GL_2}_{\GSpin_4} 
\left( {\St}_{\GL_2} \otimes \chi_1 \boxtimes {\St}_{\GL_2} \otimes \chi_2 \right) 
= {\St}_{\GSpin_4} \otimes \chi
\]
for some $\chi.$
We have $I^{\GSpin_4}(\ts) \s \{1\}$ as $I^{G}(\St_G) \s \{1\}.$ Thus, by \eqref{stab}, the $L$-packet remains a singleton and the restriction is irreducible.

\begin{itemize}
\item To determine $\chi,$ we use the required properties of $\chi_1, \chi_2$.  
Using 
\begin{equation}  \label{descriptionT}
T=\left\{ \left( \begin{bmatrix} a&0 \\  0&b \end{bmatrix} , \begin{bmatrix} c&0 \\  0&d \end{bmatrix} \right) \, 
\middle| \, ab = cd \right\}, 
 \end{equation}
we have 
$
\chi_1(ab)=\chi_2(cd) ~~ \Leftrightarrow ~~ \chi_1=\chi_2. 
$
Denote $\chi_1 = \chi_2$ by $\chi$. 
\end{itemize}

For \eqref{generici}, we have 
\[
N_{{\GL}_2(\CC) \times {\GL}_2(\CC)} = 
\left( \begin{bmatrix} 0&1 \\  0&0 \end{bmatrix} , \begin{bmatrix} 0&1 \\  0&0 \end{bmatrix} \right) 
\overset{\eqref{indentity}}{\Longleftrightarrow}
N_{{\GSO}_4(\CC)}= 
\begin{bmatrix} 
0&1&1&0 \\ 
0&0&0&-1 \\
0&0&0&-1 \\
0&0&0&0 
\end{bmatrix} 
\]

\item Next we consider 
\begin{equation} \label{genricii}
{\Res}^{\GL_2 \times \GL_2}_{\GSpin_4} \left( i_{\GL_1 \times \GL_1}^{\GL_2}(\chi_1 \otimes \chi_2) \boxtimes {\St}_{\GL_2} \otimes \chi \right).
\end{equation}
By \eqref{stab}, the fact that $\ts_2 \not\s \ts_1\teta$ for any $\teta \in (F^{\times})^D$, and since $I^{G}(\St_G) \s \{1\}$,  
it follows that 
\[
I^{\GSpin_4}(\ts) \s \{1\}.
\]
Thus, the $L$-packet remains a singleton and the restriction \eqref{genricii} is irreducible.

\begin{itemize}
\item To describe the restriction \eqref{genricii}, we proceed similarly as above.  We have 
\[
\chi_1(a)\chi_2(b)=\chi(cd)=\chi(ab) ~~ \Leftrightarrow ~~ \chi_1\chi^{-1}(a) =\chi_2^{-1}\chi(b)
\]
Specializing to $a=b$ and $c=d$ in the center, we have
\[
\chi_1\chi_2\chi^{-2}=1
\] 
\end{itemize}
For \eqref{genricii} , we have 
\[
N_{{\GL}_2(\CC) \times {\GL}_2(\CC)} = 
\left( \begin{bmatrix} 0&0 \\  0&0 \end{bmatrix} , \begin{bmatrix} 0&1 \\  0&0 \end{bmatrix} \right) 
\overset{\eqref{indentity}}{\Longleftrightarrow}
N_{{\GSO}_4(\CC)}= 
\begin{bmatrix} 
0&0&1&0 \\ 
0&0&0&-1 \\
0&0&0&0 \\
0&0&0&0 
\end{bmatrix}.  
\]

\item We consider
\[
{\Res}^{\GL_2 \times \GL_2}_{\GSpin_4} \left( i_{\GL_1 \times \GL_1}^{\GL_2}(\chi_1 \otimes \chi_2) \boxtimes i_{\GL_1 \times \GL_1}^{\GL_2}(\chi_3 \otimes \chi_4) \right) 
= i_{T}^{\GSpin_4} \left( \chi_1 \otimes \chi_2 , \chi_3 \otimes  \chi_1 \chi_2 \chi_3^{-1} \right).
\]
Here, $\chi_1 \neq \chi_2 \nu^{\pm1}$ and $\chi_3 \neq \chi_4 \nu^{\pm1}.$ 
Note that by \eqref{stab} this induced representation may be irreducible or consist of two irreducible inequivalent constituents.
We have 
\[
N_{{\GL}_2(\CC) \times {\GL}_2(\CC)} = 
\left( \begin{bmatrix} 0&0 \\  0&0 \end{bmatrix} , \begin{bmatrix} 0&0 \\  0&0 \end{bmatrix} \right) 
\overset{\eqref{indentity}}{\Longleftrightarrow}
N_{{\GSO}_4(\CC)}=
\begin{bmatrix} 
0&0&0&0 \\ 
0&0&0&0 \\
0&0&0&0 \\
0&0&0&0 
\end{bmatrix}.  
\]

\item Given a supercuspidal $\ts \in \Irr(\GL_2)$, we consider
\begin{equation} \label{iv}
{\Res}^{\GL_2 \times \GL_2}_{\GSpin_4} \left( \ts \boxtimes {\St}_{\GL_2} \otimes \chi \right).
\end{equation}
Since $I^{G}(\St_G) \s \{1\},$ due to \eqref{stab}, the restriction \eqref{iv} is irreducible.
We then have 
\[
N_{\GL_2(\CC) \times \GL_2(\CC)} = 
\left( \begin{bmatrix} 0&0 \\  0&0 \end{bmatrix} , \begin{bmatrix} 0&1 \\  0&0 \end{bmatrix} \right) 
\overset{\eqref{indentity}}{\Longleftrightarrow}
N_{{\GSO}_4(\CC)} = 
\begin{bmatrix} 
0&0&1&0 \\ 
0&0&0&-1 \\
0&0&0&0 \\
0&0&0&0 
\end{bmatrix}.  
\]

\item Given supercuspidal $\ts \in \Irr(\GL_2),$ we next consider
\[
{\Res}^{\GL_2 \times \GL_2}_{\GSpin_4} \left(\ts \boxtimes i_{\GL_1 \times \GL_1}^{\GL_2}(\chi_1 \otimes \chi_2) \right).
\]
Note from  \eqref{stab} that this may be irreducible or consist of two irreducible inequivalent constituents.
We have 
\[
N_{{\GL}_2(\CC) \times {\GL}_2(\CC)} = 
\left( \begin{bmatrix} 0&0 \\  0&0 \end{bmatrix} , \begin{bmatrix} 0&0 \\  0&0 \end{bmatrix} \right) 
\overset{\eqref{indentity}}{\Longleftrightarrow}
N_{{\GSO}_4(\CC)} = 0_{4\times4}. 
\]

\end{enumerate}

\subsection{Non-Generic Representations of $\GSpin_4$} \label{nongeneric} 
If $\sigma \in \Irr(\GSpin_4)$ is non-generic, then $\sigma$ is of the form 
\begin{equation} \label{non-generic form}
{\Res}^{\GL_2 \times \GL_2}_{\GSpin_4} \left( (\chi \circ \det) \boxtimes  \ts \right), 
\end{equation}
with $\ts \in \Irr(\GL_2).$ Note this restriction is irreducible due to \eqref{stab}, and that as $\chi \circ \det$ is non-generic, so is the restriction $\sigma$ for any  $\ts \in \Irr(\GL_2).$

For $\tsigma=\St \in \Irr(\GL_2),$ we have 
\[
N_{{\GL}_2(\CC) \times {\GL}_2(\CC)} 
= 
\left( \begin{bmatrix} 0&0 \\  0&0 \end{bmatrix} , \begin{bmatrix} 0&1 \\  0&0 \end{bmatrix} \right) 
\overset{\eqref{indentity}}{\Longleftrightarrow}
N_{{\GSO}_4(\CC)} = 
\begin{bmatrix} 
0&0&1&0 \\ 
0&0&0&-1 \\
0&0&0&0 \\
0&0&0&0 
\end{bmatrix},  
\]
and otherwise we have 
\[
N_{{\GL}_2(\CC) \times {\GL}_2(\CC)} = 
\left( \begin{bmatrix} 0&0 \\  0&0 \end{bmatrix} , \begin{bmatrix} 0&0 \\  0&0 \end{bmatrix} \right) 
\overset{\eqref{indentity}}{\Longleftrightarrow}
 N_{{\GSO}_4(\CC)} = 0_{4\times4}. 
\]

We summarize the above information about the representations of $\GSpin_4$ in Table \ref{g4-maintable}.

\subsection{Computation of the Adjoint $L$-function for $\GSpin_4$} 
We now give explicit expressions for the adjoint $L$-function for each of the representations of $\GSpin_4(F)$.  
We start by recalling that the adjoint $L$-functions of the representations $\ts \in \Irr(\GL_2)$ are as follows. 
\[ 
L(s, \ts, \Ad_2) = 
 \begin{cases}
L(s)^2 L(s, \chi_1\chi_2^{-1}) L(s, \chi_1^{-1}\chi_2), & \mbox{ if } \ts=i_{\GL_1 \times \GL_1}^{\GL_2}(\chi_1 \boxtimes \chi_2) \mbox{ with } \chi_1 \chi_2^{-1} \neq \nu^{\pm1}; \\ 
L(s)L(s+1), &~~\text{if } \ts=\St_{\GL_2} \otimes \chi; \\
L(s) L(s, \ts, \Sym^2 \otimes \omega_{\ts}^{-1}), & \mbox{ if } \ts \mbox{ is supercuspidal}; \\  
L(s)^2 L(s-1)L(s+1), & \mbox{ if } \ts=\chi \circ \det.
\end{cases}
\]
Here, $L(s) = L(s, 1_{F^\times})$.  Recall our choice of notation  
\[ L(s, \ts, \Ad_2) = L(s) L(s, \ts, \Ad). \]

Combining with \eqref{Ad-L}, Sections \ref{irr parameter gspin4} and \ref{non-irr parameter gspin4}, we have the following. 

\begin{enumerate}
\item[$\mathfrak{gnr}$-(a)\&(b)] Given a supercuspidal $\sigma \in \Irr(\GSpin_4),$ we recall that 
\[
\sigma \subset  {\Res}^{\GL_2 \times \GL_2}_{\GSpin_4}(\ts_1 \boxtimes \ts_2) 
\]
for some supercuspidal $\ts_1 \boxtimes \ts_2 \in \Irr(\GL_2 \times \GL_2).$
By \eqref{Ad-L2} we have
\[ 
L(s, \sigma, \Ad) = 
L(s, \tsigma_1, \Sym^2 \otimes \omega_{\ts_1}^{-1}) L(s, \tsigma_2, \Sym^2 \otimes \omega_{\ts_2}^{-1}).
\]

\end{enumerate}

\begin{enumerate}[$\mathfrak{gnr}$-(i)]

\item Given 
\[
\sigma=\St_{\GSpin_4} \otimes \chi \in \Irr(\GSpin_4),
\]
by \eqref{Ad-L2} we have
\[
L(s, \sigma, \Ad) = L(s+1)^2.
\]

\item Given $\sigma \in \Irr(\GSpin_4)$ such that 
\[
\sigma = {\Res}^{\GL_2 \times \GL_2}_{\GSpin_4} \left( i_{\GL_1 \times \GL_1}^{\GL_2}(\chi_1 \otimes \chi_2) \boxtimes {\St}_{\GL_2} \otimes \chi \right), 
\]
by \eqref{Ad-L2} we have
\[
L(s, \sigma, \Ad)= L(s) L(s, \chi_1\chi_2^{-1}) L(s, \chi_1^{-1}\chi_2) L(s+1).
\]

\item Given $\sigma \in \Irr(\GSpin_4)$ such that 
\[
\sigma \subset {\Res}^{\GL_2 \times \GL_2}_{\GSpin_4} \left( i_{\GL_1 \times \GL_1}^{\GL_2}(\chi_1 \otimes \chi_2) \boxtimes i_{\GL_1 \times \GL_1}^{\GL_2}(\chi_3 \otimes \chi_4) \right)
\]
by \eqref{Ad-L2} we have
\[
 L(s, \sigma, \Ad) = L(s)^2 
 L(s, \chi_1\chi_2^{-1}) L(s, \chi_1^{-1}\chi_2) 
 L(s, \chi_3\chi_4^{-1}) L(s, \chi_3^{-1}\chi_4).
\]

\item Given $\sigma \in \Irr(\GSpin_4)$ such that 
\[
\sigma ={\Res}^{\GL_2 \times \GL_2}_{\GSpin_4} \left( \ts \boxtimes {\St}_{\GL_2} \otimes \chi \right)
\]
by \eqref{Ad-L2} we have
\[
 L(s, \sigma, \Ad) = 
 L(s, \tsigma_2, \Sym^2 \otimes \omega_{\ts_2}^{-1}) 
L(s+1). 
\]

\item Given $\sigma \in \Irr(\GSpin_4)$ such that 
\[
\sigma \subset {\Res}^{\GL_2 \times \GL_2}_{\GSpin_4} \left( \ts \boxtimes i_{\GL_1 \times \GL_1}^{\GL_2}(\chi_1 \otimes \chi_2) \right)
\]
by \eqref{Ad-L2} we have
\[
L(s, \sigma, \Ad) = 
L(s)
L(s, \tsigma_2, \Sym^2 \otimes \omega_{\ts_2}^{-1}) 
L(s, \chi_1\chi_2^{-1})L(s, \chi_1^{-1}\chi_2).
\]

\end{enumerate}

\begin{enumerate}
\item[$\mathfrak{nongnr}$]
Given a non-generic $\sigma \in \Irr(\GSpin_4),$ from \eqref{non-generic form}, we recall that 
\[
\sigma =  {\Res}^{\GL_2 \times \GL_2}_{\GSpin_4} (\chi\circ\det \,\boxtimes\, \tsigma)
\]
and by \eqref{Ad-L2} we have
\[
L(s, \sigma, \Ad) = 
L(s) L(s-1) L(s+1) 
L(s, \tsigma, \Ad).
\]

\end{enumerate}

We summarize the explicit computations above in Table \ref{g4-polestable}.

\section{Representations of $\GSpin_6$} \label{genericclassification6}

We now list all the representations of $\GSpin_6(F)$ and 
then calculate their associated adjoint $L$-function explicitly.  
Again, we do this explicit calculation by finding the $6 \times 6$ 
nilpotent matrix in the complex dual group $\GSO_6(\CC)$ in each case that is 
associated with the parameter of the representation.

\subsection{The Represenations}

\subsubsection{Classification of representations of $\GSpin_6$} 
Again, following \cite{acgspin}, we have
\begin{equation} 
1 \longrightarrow {\GSpin}_6(F) \longrightarrow {\GL}_1(F) \times {\GL}_4(F) 
\longrightarrow F^{\times} \longrightarrow 1.  
\end{equation} 
Recall that 
\begin{equation} \label{convenient description of GSpin6(F)}
{\GSpin}_6(F) \s \left\{ (g_1, g_2) \in {\GL}_1(F) \times {\GL}_4(F) : g_1^{2} = \det g_2 \right\},  
\end{equation}
\begin{equation} \label{L-gp for GSpin6}
{^L{\GSpin}_6}  = \widehat{{\GSpin}_6}  = {\GSO}_{6}(\CC)  \s  ({\GL}_1(\CC) \times {\GL}_4(\CC))  / \{ (z^{-2}, z) : z \in \CC^{\times} \}, 
\end{equation}
and 
\begin{equation} \label{ex seq L-gp for GSpin6}
1 \longrightarrow \CC^{\times} \longrightarrow {\GL}_1(\CC) \times {\GL}_4(\CC) 
\overset{pr_6}{\longrightarrow} 
\widehat{{\GSpin}_6} \longrightarrow 1. 
\end{equation}
Just as the rank two case, here too we view 
$\GSO_6$ as the group similitude orthogonal $6 \times 6$ matrices 
with respect to the analogous $6\times6$, anti-diagonal, matrix $J=J_6$ 
as in \eqref{J-matrix}, and similarly define its Lie algebra with respect to $J$.

\subsubsection{Construction of the $L$-packets of $\GSpin_6$ (recalled from \cite{acgspin})}
Given $\sigma \in \Irr(\GSpin_6)$ we have a lift $\ts \in \Irr({\GL}_1 \times {\GL}_4)$ such that 
\[
\sigma \hookrightarrow {\Res}_{{\GSpin}_{6}}^{\GL_1 \times \GL_4}(\ts).
\] 
It follows from the LLC for $GL_n$ \cite{ht01, he00, scholze13} that there is a unique $\tvp_{\ts} \in \Phi({\GL}_1 \times {\GL}_4)$ 
corresponding to the representation $\ts.$ 
We now have a surjective, finite-to-one map 
\begin{eqnarray} \label{L map for GSpin6}
{\L}_{6} : {\Irr}({\GSpin}_{6}) & \longrightarrow & \Phi({\GSpin}_{6}) \\
\sigma & \longmapsto & pr_6 \circ \tvp_{\ts},  \nonumber
\end{eqnarray} 
which does not depend on the choice of the lifting $\ts.$
Then, for each $\vp \in \Phi({\GSpin}_{6}),$
all inequivalent irreducible constituents of $\ts$ constitutes the $L$-packet
\begin{equation} \label{def of L-packet for GSpin6}
\Pi_{\vp}({\GSpin}_{6}):=\Pi_{\ts}({\GSpin}_6) = 
\left\{ \sigma : \sigma \hookrightarrow 
{\Res}_{{\GSpin}_{6}}^{\GL_1 \times \GL_4}(\ts) \right\} \Big\slash \s, 
\end{equation}
where $\ts$ is the unique member of $\Pi_{\tvp}(\GL_1 \times \GL_4)$ and 
$\tvp \in \Phi(\GL_1 \times \GL_4)$ is such that $pr_{6} \circ \tvp=\vp.$ 
We note that the construction does not depends on the choice of $\tvp$.   
Further details can be found in \cite[Section 6.1]{acgspin}.

Following \cite[Section 6.3]{acgspin},  given $\vp \in \Phi(\GSpin_6),$ fix the lift 
\[ 
\tvp=\teta \otimes \tvp_0 \in \Phi({\GL}_1\times {\GL}_4) 
\] 
with $\tvp_0 \in \Phi({\GL}_4)$ such that $\vp = pr_6 \circ \tvp$. 
Let 
\[ 
\ts = \teta \boxtimes \ts_0 \in \Pi_{\tvp}({\GL}_1\times {\GL}_4) 
\] 
be the unique member such that $\{\ts_0\} = \Pi_{\tvp_0}({\GL}_4).$

Recall that 
\[
I^{\GSpin_6}(\ts) := \left\{ \tchi \in \Big({\GL}_1(F) \times {\GL}_4(F) /{\GSpin}_6(F)\Big)^D : \ts \otimes \tchi \s \ts  \right\}.
\]
Then we have
\begin{equation} \label{1to1 gspin6}
\Pi_{\vp}({\GSpin}_6) \, \overset{1-1}{\longleftrightarrow} \,I^{\GSpin_6}(\ts),
\end{equation}
and by \cite[Lemma 6.5 and Proposition 6.6]{acgspin} we have
\begin{equation} \label{stab gspin6}
I^{\GSpin_6}(\ts) \s
\{
\tchi \in I^{\SL_4}(\ts_0) : \tchi^2 = 1_{F^\times} 
\}
\end{equation}
and any $\tchi \in I^{\GSpin_6}(\ts)$ is of the form 
\[
\tchi = (\tchi')^{-2} \boxtimes \tchi',
\]
for some $\tchi' \in (F^{\times})^D.$

\subsection{Generic Representations of $\GSpin_6$} \label{generic section gspin6}

Thanks to the group structure \eqref{convenient description of GSpin6(F)} and the relation of generic representations in 
Section \ref{restrictionofgeneric}, in order to classify the generic representations of $\GSpin_6,$ 
it suffices to classify the generic representations of $GL_4$.

Here are two key facts from the $\GL$ theory.
\begin{itemize}
\item 
Recall from \cite[Theorem 9.7]{zel80} and \cite[Theorem 2.3.1]{ku94} that 
a generic representation of $\GL_4$ is of the form
\[
i_{M_{\flat}}^{GL_4} (\sigma_\flat)
\]
where $M_\flat$ runs through any $F$-Levi subgroup of $\GL_4$ (including $\GL_4$ itself) 
and $\sigma_\flat$ is any essentially square-integrable representation of $M_\flat.$

\item For their $L$-parameters, we note from \cite[\S 5.2]{ku94} that the generic representations of $GL_4$ 
have Langlands parameters (i.e., 4-dimensional Weil-Deligne representations $(\rho, N)$) of the form
\[
(\rho_1 \otimes sp(r_1)) \otimes .. \otimes (\rho_t \otimes sp(r_t))
\] 
with $t \leq 4,$ 
where $\rho_i$'s are irreducible and no two segments are linked.
\end{itemize}

\subsubsection{Irreducible Parameters} \label{irr parameter gspin6}
Let $\vp \in \Phi(\GSpin_6)$ be irreducible.  
Then $\tvp$ and $\tvp_0$ are also irreducible. By Section \ref{restrictionofgeneric}, we have the following.
\begin{pro} 
Let $\vp \in \Phi(\GSpin_6)$ be irreducible. Every member in $\Pi_{\vp}(\GSpin_6)$ is supercuspidal and generic.
\end{pro}
To see the internal structure of $\Pi_{\vp}(\GSpin_6),$ we need, by \eqref{1to1 gspin6}, to know  
the detailed structure of $I^{\GSpin_6}(\ts)$ as follows.  

\begin{enumerate}[$\mathfrak{gnr}$-(a)]
\item 
Given $\sigma \in \Irr_{\rm sc}(\GSpin_6),$ we have 
\begin{equation} \label{sc gl4}
\tsigma = \tsigma_0 \boxtimes \teta \in {\Irr}_{\rm sc}(\GL_4 \times \GL_1).
\end{equation} 
From \cite[Proposition 2.1]{acgspin}, we recall the identification:
\begin{equation} \label{indentitygspin6}
\Delta^\vee = \left\{ \beta^\vee_1 = f^*_2 - f^*_3, \beta^\vee_2 = f^*_1 - f^*_2, \beta^\vee_3 = f^*_3 - f^*_4 \right\}.
\end{equation}
using the notation $f_{ij}$ and $f^*_{ij},$ $1 \le i, j \le 4,$ for the usual $\ZZ$-basis of characters and cocharacters of  $\GL_4$. 
Also, $ \{ \beta_1, \beta_2, \beta_3 \} $ are the simple roots of $\GSpin_6$. 

We have 
\[
N_{ {\GL}_4(\CC) \times {\GL}_1(\CC)} 
= 
\left(
0_{4\times4}, 
0 \right) 
\overset{\eqref{indentitygspin6}}{\Longleftrightarrow} 
N_{{\GSO}_6(\CC)} = 0_{6\times6}. 
\]

\end{enumerate}

\subsubsection{Reducible Parameters} \label{non-irr parameter gspin6}
When $\tvp_0$ is not irreducible, we have proper parabolic inductions.  
An exhaustive list of $F$-Levi subgroups $\bM$ of $\GSpin_6$ (up to isomorphism) is as follows.
\begin{itemize}

\item $\bM \s \GL_1 \times GL_1 \times \GL_1 \times \GL_1 = \tbM \cap \GSpin_6$,  
where $\tbM = (\GL_1 \times \GL_1 \times \GL_1 \times \GL_1) \times \GL_1.$  

\item $\bM \s \GL_2 \times \GL_1 \times \GL_1 = \tbM \cap \GSpin_6$, 
where $\tbM = ( \GL_2 \times \GL_1 \times \GL_1) \times \GL_1.$  

\item $\bM \s \GL_3 \times \GL_1= \tbM \cap \GSpin_6$, 
where $\tbM = (\GL_3 \times \GL_1) \times \GL_1.$   
(Note: The factor $\GL_1$ of $\bM$ is $\GSpin_0$  by convention.)

\item $\bM \s \GL_1 \times \GSpin_4 = \tbM \cap \GSpin_6$, 
where $\tbM = (\GL_2 \times \GL_2) \times \GL_1.$  

\item $\bM \s \GSpin_6 = \tbM \cap \GSpin_6,$ 
where $\tbM = \GL_4 \times \GL_1.$ 

\end{itemize}
(Note that $\bM \s \GL_2 \times \GL_2$ does not occur on this list.)  
We now consider each case and, by abuse of notation, conflate algebraic groups 
and their $F$-points.

\begin{enumerate}[$\mathfrak{gnr}$-(I)]
\item 
$\bM \s \GL_1 \times GL_1 \times \GL_1 \times \GL_1$ and 
$\tbM = (\GL_1 \times \GL_1 \times \GL_1 \times \GL_1) \times \GL_1$. 

Given $\chi_i \in (F^\times)^D$ we consider 
\begin{equation} \label{ps gspin6}
i_{M}^{\GSpin_6}(\chi_1 \boxtimes \chi_2 \boxtimes \chi_3 \boxtimes \chi_4).
\end{equation}
Write   
$
\chi_1 \boxtimes \chi_2 \boxtimes \chi_3 \boxtimes \chi_4 = (\tchi_1 \boxtimes \tchi_2 \boxtimes \tchi_3 \boxtimes \tchi_4 \boxtimes \teta)|_{M}
$
with $\tchi_i, \teta \in (F^\times)^D$ so that
\[
\tchi_1 \tchi_2 \tchi_3 \tchi_4 = \teta^2.
\]
Then we have the following relations  
\begin{equation} \label{chi-tchi}
\chi_1=\tchi_1,~ \chi_2=\tchi_2, ~ \chi_3=\tchi_3, ~ \chi_4=\teta^2(\tchi_2 \tchi_3 \tchi_4)^{-1}.
\end{equation}

By Section \ref{restrictionofgeneric}, we know that the representation \eqref{ps gspin6} 
is generic if and only if its lift 
\begin{equation} \label{ps gl1x4} 
i_{\tM}^{\GL_4 \times GL_1} (\tchi_1 \boxtimes \tchi_2 \boxtimes \tchi_3 \boxtimes \tchi_4 \boxtimes \teta)
\end{equation}
is generic if and only if 
\begin{equation} \label{ps gl4} 
i_{\GL_1 \times \GL_1 \times \GL_1 \times \GL_1}^{GL_4} 
(\tchi_1 \boxtimes \tchi_2 \boxtimes \tchi_3 \boxtimes \tchi_4)
\end{equation}
is generic. 
By the classification of the generic representations of $\GL_n$ 
(\cite[Theorem 9.7]{zel80} and \cite[Theorem 2.3.1]{ku94}), 
this amounts to \eqref{ps gl4} being irreducible.  
By \cite[Theorem 2.1.1]{ku94} and \cite{bz77, zel80}, the necessary and sufficient 
condition for this to occur is that there is no pair $i, j$ with $i \neq j$ such that
\[
\tchi_i = \nu \tchi_j. 
\]

We have 
\[
N_{ {\GL}_4(\CC) \times {\GL}_1(\CC)}=\left(
0_{4\times4}, 
0 \right) 
\overset{\eqref{indentitygspin6}}{\Longleftrightarrow} N_{{\GSO}_6(\CC)} =
0_{6\times6} 
\]

\item $\bM \s \GL_2 \times \GL_1 \times \GL_1$ and 
$\tbM = ( \GL_2 \times \GL_1 \times \GL_1) \times \GL_1$. 

Given $\sigma_0\in \Irr_{\rm esq}(\GL_2)$ and  $\chi_1,\chi_2 \in (F^\times)^D$, 
we consider 
\begin{equation} \label{ps gspin6 II}
i_{M}^{\GSpin_6}(\sigma_0 \boxtimes \chi_1 \boxtimes \chi_2).
\end{equation}
Write 
$
\sigma_0 \boxtimes \chi_1 \boxtimes \chi_2 = (\tsigma_0 \boxtimes \tchi_1 \boxtimes \tchi_2 \boxtimes \teta)|_{M}
$
with $\tsigma_0 \in \Irr_{\rm esq}(\GL_2), \tchi_i, \teta \in (F^\times)^D$.

Given $(g,h_1,h_2,h_3) \in \tM$ with $\det (g h_1h_2) = h_3^{2},$
\begin{itemize}
\item if we set $(g, h_1, h_3) \in M$, we have
\begin{eqnarray*}
\tsigma_0(g) \tchi_1(h_1) \tchi_2(h_2) \teta(h_3) & = & 
\tsigma_0(g) \tchi_1(h_1) \tchi_2(\det g^{-1} h_1^{-1}h_3^{2}) \teta(h_3) \\ 
& =& (\tsigma_0 \tchi^{-1}_2 \circ \det)(g) (\tchi_1\tchi_2^{-1})(h_1)(\tchi_2^2\teta)(h_3) \\ 
& = & \sigma(g)\chi_1(h_1)\chi_2(h_3).
\end{eqnarray*}
Then we have
\[
\tsigma_0 =\sigma_0 \tchi_2,~~ \tchi_1=\chi_1 \tchi_2,~~\teta= \chi_2 \tchi_2^{-2}.
\]
\item If we set $(g, h_2, h_3) \in M$, we have
\begin{eqnarray*} 
\tsigma_0(g)\tchi_1(h_1)\tchi_2(h_2)\teta(h_3) & = & 
\tsigma_0(g)\tchi_1(\det g^{-1} h_2^{-1}h_3^{2})\tchi_2(h_2)\teta(h_3) \\ 
& = & (\tsigma_0 \tchi^{-1}_1 \circ \det)(g) (\tchi_2\tchi_1^{-1})(h_2)(\tchi_1^2\teta)(h_3) \\ 
& = & \sigma(g)\chi_1(h_2)\chi_2(h_3).
\end{eqnarray*} 
Then we have
\begin{equation} \label{case 2}
\tsigma_0 =\sigma_0 \tchi_1,~~ \tchi_2=\chi_2\tchi_1,~~\teta= \chi_1 \tchi_1^{-2}.
\end{equation}
\end{itemize}

As before, the representation \eqref{ps gspin6 II} is generic if and only if its lift 
\begin{equation} \label{ps gl1x4 II} 
i_{\tM}^{\GL_4 \times GL_1} (\tsigma_0 \boxtimes \tchi_1 \boxtimes \tchi_2 \boxtimes \teta)
\end{equation}
is generic if and only if 
\begin{equation} \label{ps gl4 II} 
i_{\GL_2 \times \GL_1 \times \GL_1}^{GL_4} (\tsigma_0 \boxtimes \tchi_1 \boxtimes \tchi_2)
\end{equation}
is generic. 
Again by the classification of the generic representations of $\GL_n$ 
this amounts to \eqref{ps gl4 II} being irreducible.  
Hence, we must have
\[
\tchi_1 \neq \nu^{\pm1} \tchi_2.
\]
In other words, given $(g,h_1,h_2,h_3) \in \tM$ with $\det (g h_1h_2) = h_3^{2},$ 
\begin{itemize}
\item if we set $(g, h_1, h_3) \in M,$ then 
\[
\chi_1 \neq \nu^{\pm1};
\]
\item if we set $(g, h_2, h_3) \in M,$ then 
\[
\chi_2 \neq \nu^{\pm1}.
\]
\end{itemize}\
We have the following two cases. 
If $\sigma_0$ is supercuspidal, then 
\[
N_{ {\GL}_4(\CC) \times {\GL}_1(\CC)} = 
\left(0_{4\times4}, 
0 \right) 
\overset{\eqref{indentitygspin6}}{\Longleftrightarrow} 
N_{{\GSO}_6(\CC)} = 0_{6\times6}.  
\]

If $\sigma_0$ is non-supercuspidal, then 
\[
N_{ {\GL}_4(\CC) \times {\GL}_1(\CC)} = 
\left(
\begin{bmatrix} 
0&1&0&0 \\  
0&0&0&0 \\  
0&0&0&0 \\  
0&0&0&0
\end{bmatrix}, 0 \right) 
\overset{\eqref{indentitygspin6}}{\Longleftrightarrow} 
N_{{\GSO}_6(\CC)} =
\begin{bmatrix} 
0&0&0&0&0&0 \\ 
0&0&1&0&0&0 \\ 
0&0&0&0&0&0 \\ 
0&0&0&0&-1&0 \\ 
0&0&0&0&0&0 \\ 
0&0&0&0&0&0 
\end{bmatrix}.  
\]

\item $\bM \s \GL_3 \times \GL_1$ and 
$\tbM = (\GL_3 \times \GL_1) \times \GL_1.$

Given $\sigma_0 \in \Irr_{\rm esq}(\GL_3)$ and $\chi \in (F^\times)^D$, we consider 
\begin{equation} \label{ps gspin6 III}
i_{M}^{\GSpin_6}(\sigma_0 \boxtimes \chi).
\end{equation}
Write 
$
\sigma_0 \boxtimes \chi = (\tsigma_0 \boxtimes \tchi \boxtimes \teta)|_{M}
$ 
with $\tsigma_0 \in \Irr_{\rm esq}(\GL_3), \tchi, \teta \in (F^\times)^D.$

Given $(g,h_1,h_2) \in \tM$ with $\det(g h_1) = h_2^{2}$,  
if we set $(g,h_2) \in M$, then we have
\begin{eqnarray} \label{case 3}
\tsigma_0(g)\tchi(h_1)\teta(h_2) 
& = & \tsigma_0(g)\tchi(\det g^{-1} h_2^{2})\teta(h_2) \\ 
\nonumber
& = & (\tsigma_0 \tchi^{-1} \circ \det)(g) (\tchi^2\teta)(h_2) \\ 
\nonumber
& = & \sigma(g)\chi(h_2).
\end{eqnarray}
Then, we have
\[
\tsigma_0 =\sigma_0 \tchi \quad \mbox{ and } \quad \teta= \chi_2 \tchi^{-2}.
\]

As before, \eqref{ps gspin6 III} is generic if and only if its lift  
\begin{equation} \label{ps gl1x4 III} 
i_{\tM}^{\GL_4 \times GL_1} (\tsigma_0 \boxtimes \tchi \boxtimes \teta)
\end{equation}
is generic if and only if 
\begin{equation} \label{ps gl4 III} 
i_{\GL_3 \times \GL_1}^{GL_4} (\tsigma_0 \boxtimes \tchi)
\end{equation}
is generic. 
This amounts to \eqref{ps gl4 III} being irreducible as before, 
which is always true since $\tsigma_0$ is an essentially square integrable 
representation of $\GL_3$. 
Note that by the classification of essentially square-integrable representations of 
$\GL_3$ (\cite[Proposition 1.1.2]{ku94}), $\tsigma_0$ must be either supercuspidal or 
the unique subrepresentation of
\begin{equation} \label{3gl1}
i_{\GL_1 \times \GL_1 \times \GL_1}^{\GL_3} \left( \nu \chi \boxtimes \chi \boxtimes \nu^{-1} \chi \right)
\end{equation}
with any $\chi \in (F^\times)^D.$

We have the following two cases. If $\sigma_0$ is supercuspidal, then 
\[
N_{ {\GL}_4(\CC) \times {\GL}_1(\CC)} = 
\left( 0_{4\times4}, 
0 \right) 
\overset{\eqref{indentitygspin6}}{\Longleftrightarrow} 
N_{{\GSO}_6(\CC)} = 
0_{6\times6}.  
\]

If $\sigma_0$ is the non-supercuspidal, unique, subrepresentation of \eqref{3gl1}, then 
\[
N_{ {\GL}_4(\CC) \times {\GL}_1(\CC)} = 
\left(
\begin{bmatrix} 
0&1&0&0 \\  
0&0&1&0 \\  
0&0&0&0 \\  
0&0&0&0
\end{bmatrix}, 0 \right) 
\overset{\eqref{indentitygspin6}}{\Longleftrightarrow} N_{{\GSO}_6(\CC)} =
\begin{bmatrix} 
0&1&0&0&0&0 \\ 
0&0&1&0&0&0 \\ 
0&0&0&0&0&0 \\ 
0&0&0&0&-1&0 \\ 
0&0&0&0&0&-1 \\ 
0&0&0&0&0&0 
\end{bmatrix}.  
\]

\item $\bM \s \GL_1 \times \GSpin_4$ and $\tbM = (\GL_2 \times \GL_2) \times \GL_1.$

Given $\sigma_0 \in \Irr_{\rm esq}(\GSpin_4)$ and $\chi \in (F^\times)^D$ we consider 
\begin{equation} \label{ps gspin6 IV}
i_{M}^{\GSpin_6}( \chi \boxtimes \sigma_0).
\end{equation}
Write 
$
\chi \boxtimes \sigma_0  \subset (\tsigma_1 \boxtimes \tsigma_2 \boxtimes  \teta)|_{M}
$ 
with $\tsigma_i \in \Irr_{\rm esq}(\GL_2), \teta \in (F^\times)^D.$

As before, \eqref{ps gspin6 IV} is generic if and only if its lift  
\begin{equation} \label{ps gl1x4 IV} 
i_{\tM}^{\GL_4 \times GL_1} (\tsigma_1 \boxtimes \tsigma_2 \boxtimes \teta)
\end{equation}
is generic if and only if 
\begin{equation} \label{ps gl4 IV} 
i_{\GL_2 \times \GL_2}^{GL_4} (\tsigma_1 \boxtimes \tsigma_2)
\end{equation}
is generic. 
This amounts to \eqref{ps gl4 IV} being irreducible.  Thus, we must have
\[
\tsigma_1 \neq  \nu^{\pm1} \tsigma_2.
\]
We have several cases to consider.  
If $\sigma_0$ is supercuspidal (so are $\tsigma_1$ and $\tsigma_2$), then 
\[
N_{ {\GL}_4(\CC) \times {\GL}_1(\CC)}=\left(
0_{4\times4} 
0 \right) 
\overset{\eqref{indentitygspin6}}{\Longleftrightarrow} 
N_{{\GSO}_6(\CC)} = 0_{6\times6}.  
\]

If $\sigma_0$ is non-supercuspidal,  then 
for supercuspidal $\tsigma_1$ and non-supercuspidal $\tsigma_2$ we have 
\[
N_{ {\GL}_4(\CC) \times {\GL}_1(\CC)} = 
\left(
\begin{bmatrix} 
0&0&0&0 \\  
0&0&0&0 \\  
0&0&0&1 \\  
0&0&0&0
\end{bmatrix}, 0 \right) 
\overset{\eqref{indentitygspin6}}{\Longleftrightarrow} 
N_{{\GSO}_6(\CC)} =
\begin{bmatrix} 
0&0&0&0&0&0 \\ 
0&0&0&1&0&0 \\ 
0&0&0&0&-1&0 \\ 
0&0&0&0&0&0 \\ 
0&0&0&0&0&0 \\ 
0&0&0&0&0&0 
\end{bmatrix};  
\]
for non-supercuspidal $\tsigma_1$ and supercuspidal $\tsigma_2$ we have 
\[
N_{ {\GL}_4(\CC) \times {\GL}_1(\CC)} = 
\left(
\begin{bmatrix} 
0&1&0&0 \\  
0&0&0&0 \\  
0&0&0&0 \\  
0&0&0&0
\end{bmatrix}, 0 \right) 
\overset{\eqref{indentitygspin6}}{\Longleftrightarrow} 
N_{{\GSO}_6(\CC)} = 
\begin{bmatrix} 
0&0&0&0&0&0 \\ 
0&0&1&0&0&0 \\ 
0&0&0&0&0&0 \\ 
0&0&0&0&-1&0 \\ 
0&0&0&0&0&0 \\ 
0&0&0&0&0&0 
\end{bmatrix};  
\]
and for non-supercuspidal $\tsigma_1$ and $\tsigma_2$ we have 
\[
N_{ {\GL}_4(\CC) \times {\GL}_1(\CC)} = 
\left(
\begin{bmatrix} 
0&1&0&0 \\  
0&0&0&0 \\  
0&0&0&1 \\  
0&0&0&0
\end{bmatrix}, 0 \right) 
\overset{\eqref{indentitygspin6}}{\Longleftrightarrow} 
N_{{\GSO}_6(\CC)} =
\begin{bmatrix} 
0&0&0&0&0&0 \\ 
0&0&1&1&0&0 \\ 
0&0&0&0&-1&0 \\ 
0&0&0&0&-1&0 \\ 
0&0&0&0&0&0 \\ 
0&0&0&0&0&0 
\end{bmatrix}.  
\]

\item $\bM \s \GSpin_6$ and $\tbM = \GL_4 \times \GL_1.$  

Given $\sigma \in \Irr_{\rm esq}(\GSpin_6) \setminus \Irr_{\rm sc}(\GSpin_6)$, we consider 
\[
\sigma \subset (\tsigma \boxtimes \teta)|_{M}
\]
with $\tsigma \in \Irr_{\rm esq}(\GL_4) \setminus \Irr_{\rm sc}(\GL_4), \teta \in (F^\times)^D.$
Here, we note that $\vp \in \Phi(\GSpin_6)$ is not irreducible and neither 
$\tsigma$ nor $\sigma$ is supercuspidal. 
It is clear that $\sigma$ is generic as $\tsigma \boxtimes \teta$ is. 
By the classification of essentially square-integrable representations of 
$\GL_4$ (\cite[Proposition 1.1.2]{ku94}), $\tsigma$ must be the unique subrepresentation of either 
\begin{equation} \label{4gl1}
i_{\GL_1 \times \GL_1 \times \GL_1 \times \GL_1}^{\GL_4} 
\left( \nu^{3/2} \tchi \boxtimes \nu^{1/2} \tchi \boxtimes 
\nu^{-1/2} \tchi \boxtimes \nu^{-3/2} \tchi \right)
\end{equation}
with any $\tchi \in (F^\times)^D$ (i.e., $\ts=\St_{\GL_4} \otimes \tchi$), 
or of 
\begin{equation} \label{2gl2}
i_{\GL_2 \times \GL_2}^{\GL_4} 
\left( \nu^{1/2} \ttau \boxtimes \nu^{-1/2} \ttau \right)
\end{equation}
with any $\ttau \in \Irr_{\rm sc}(\GL_2)$.

Now, for \eqref{4gl1} we have 
\[
N_{ {\GL}_4(\CC) \times {\GL}_1(\CC)} = 
\left(
\begin{bmatrix} 
0&1&0&0 \\  
0&0&1&0 \\  
0&0&0&1 \\  
0&0&0&0
\end{bmatrix}, 0 \right) 
\overset{\eqref{indentitygspin6}}{\Longleftrightarrow} 
N_{{\GSO}_6(\CC)} =
\begin{bmatrix} 
0&1&0&0&0&0 \\ 
0&0&1&1&0&0 \\ 
0&0&0&0&-1&0 \\ 
0&0&0&0&-1&0 \\ 
0&0&0&0&0&-1\\ 
0&0&0&0&0&0 
\end{bmatrix};  
\]
and
for \eqref{2gl2} we have 
\[
N_{ {\GL}_4(\CC) \times {\GL}_1(\CC)} = 
\left(
\begin{bmatrix} 
0&0&1&0 \\  
0&0&0&1 \\  
0&0&0&0 \\  
0&0&0&0
\end{bmatrix}, 0 \right) 
\overset{\eqref{indentitygspin6}}{\Longleftrightarrow} 
N_{{\GSO}_6(\CC)} =
\begin{bmatrix} 
0&0&1&1&0&0 \\ 
0&0&0&0&0&0 \\ 
0&0&0&0&0&-1 \\ 
0&0&0&0&0&-1 \\ 
0&0&0&0&0&0 \\ 
0&0&0&0&0&0 
\end{bmatrix}.  
\]
\end{enumerate}
(We note, cf. \cite[(4.1.5)]{tate79}, that $N_{ {\GL}_4(\CC)}$ is of the form 
$O_{2 \times 2} \otimes I_{2 \times 2} + \begin{bmatrix} 0 & 1 \\ 0 & 0 \end{bmatrix} \otimes I_{2 \times 2}. $)

\subsection{Non-Generic Representaions of $\GSpin_6$} \label{nongeneric section gspin6}

Using the transitivity of the parabolic induction and the classification of generic representations of $\GL_n$,  
(\cite[Theorem 9.7]{zel80} and \cite[Theorem 2.3.1]{ku94}), the non-generic representations of 
$\GSpin_6$ are as follows.

\begin{enumerate}[$\mathfrak{nongnr}$-(A)]
\item $\bM \s \GL_1 \times GL_1 \times \GL_1 \times \GL_1$ and 
$\tbM = (\GL_1 \times \GL_1 \times \GL_1 \times \GL_1) \times \GL_1.$

Given $\chi_i \in (F^\times)^D$, by Section \ref{restrictionofgeneric} and using \eqref{chi-tchi}, 
the representation \eqref{ps gspin6} contains a non-generic constituent if and only if the same is true for 
\begin{equation} \label{ps gl1x4 non} 
i_{\tM}^{\GL_4 \times GL_1} (\tchi_1 \boxtimes \tchi_2 \boxtimes \tchi_3 \boxtimes \tchi_4 \boxtimes \teta)
\end{equation}
if and only if 
\begin{equation} \label{ps gl4 non}
i_{\GL_1 \times \GL_1 \times \GL_1 \times \GL_1}^{GL_4} 
(\tchi_1 \boxtimes \tchi_2 \boxtimes \tchi_3 \boxtimes \tchi_4)
\end{equation}
contains a non-generic constituent.  
This amounts to \eqref{ps gl4 non} being reducible. 
As before, the necessary and sufficient condition for this to occur is that 
there is some pair $i, j$ with $i \neq j$ such that
$ \tchi_i = \nu \tchi_j$.

By the Langlands classification and the description of constituents of the parabolic induction 
(see \cite[Theorem 7.1]{zel80}, \cite[Theorem 7.1]{rod82}, and  \cite[Theorems 2.1.1 \S 5.1.1]{ku94}), 
each constituent can be described as a Langlands quotient, denoted by $Q(...)$, as follows.

The first case is when there is only one pair, say 
$\tchi_1 =  \nu^{1/2} \tchi$ and $\tchi_2 = \nu^{-1/2} \tchi$ 
for some $\tchi \in (F^\times)^D$ while 
$\tchi_3 \neq  \nu^{\pm1}\tchi_j$ for $j \neq 3$ and 
$\tchi_4 \neq  \nu^{\pm1}\tchi_j$ for $j \neq 4.$ 
Then we have the non-generic constituent
\begin{equation} \label{nongenericA1}
Q\left( [ \nu^{1/2} \tchi], [ \nu^{-1/2} \tchi], [\tchi_3], [\tchi_4] \right),  
\end{equation}
which is the Langlands quotient of
\[
i_{\GL_2 \times \GL_1 \times \GL_1}^{GL_4} 
\left( Q\left( [  \nu^{1/2} \tchi ], [ \nu^{-1/2}\tchi ] \right) 
\boxtimes \tchi_3 \boxtimes \tchi_4 \right)
= 
i_{\GL_2 \times \GL_1 \times \GL_1}^{GL_4} 
\left( \left( \tchi \circ \det \right) \boxtimes \tchi_3 \boxtimes \tchi_4 \right).
\]
We have 
\[
N_{ {\GL}_4(\CC) \times {\GL}_1(\CC)}=\left(
0_{4\times4}, 
0 \right) 
\overset{\eqref{indentitygspin6}}{\Longleftrightarrow} N_{{\GSO}_6(\CC)} =
0_{6\times6}. 
\]

Note that the other constituent of this induced representation, which is generic, is 
\begin{eqnarray*}
Q\left( [ \nu^{-1/2}\tchi, \nu^{1/2} \tchi], [\tchi_3], [\tchi_4] \right) 
& = & 
i_{\GL_2 \times \GL_1 \times \GL_1}^{GL_4} 
\left( Q\left( [\nu^{-1/2} \tchi, \nu^{1/2} \tchi] \right) \boxtimes \tchi_3 \boxtimes \tchi_4 \right) \\ 
& = & 
i_{\GL_2 \times \GL_1 \times \GL_1}^{GL_4} 
\left( (\St \otimes \tchi) \boxtimes \tchi_3 \boxtimes \tchi_4 \right).
\end{eqnarray*}

The next case is when there are two pairs, say 
$\tchi_1 = \nu\tchi$, $\tchi_2 = \tchi$, and $\tchi_3 = \nu^{-1} \tchi$ 
for some $\tchi \in (F^\times)^D$ 
and $\tchi_4 \neq \nu^{\pm1} \tchi_i$ for $i = 1,2,3$.  
Then we have the following three non-generic constituents:
\begin{eqnarray} 
Q\left( [ \nu\tchi], [ \tchi], [ \nu^{-1} \tchi], [\tchi_4] \right) 
&=& i_{\GL_3 \times \GL_1}^{GL_4} ((\tchi \circ \det) \boxtimes \tchi_3 \boxtimes \tchi_4); 
\label{nongenericA2}
\\ 
Q\left( [\tchi, \nu \tchi], [\nu^{-1} \tchi], [\tchi_4] \right); && 
\label{nongenericA3} 
\\  
Q\left( [\nu \tchi], [ \tchi, \nu^{-1} \tchi], [\tchi_4] \right). && 
\label{nongenericA4} 
\end{eqnarray} 
For \eqref{nongenericA2} we have 
\[
N_{ {\GL}_4(\CC) \times {\GL}_1(\CC)} = 
\left( 0_{4\times4},  
0 \right) 
\overset{\eqref{indentitygspin6}}{\Longleftrightarrow} 
N_{{\GSO}_6(\CC)} = 0_{6\times6},  
\]
for \eqref{nongenericA3} we have 
\[
N_{ {\GL}_4(\CC) \times {\GL}_1(\CC)} = 
\left(
\begin{bmatrix} 
0&1&0&0 \\  
0&0&0&0 \\  
0&0&0&0 \\  
0&0&0&0
\end{bmatrix}, 0 \right) 
\overset{\eqref{indentitygspin6}}{\Longleftrightarrow} 
N_{{\GSO}_6(\CC)} =
\begin{bmatrix} 
0&0&0&0&0&0 \\ 
0&0&1&0&0&0 \\  
0&0&0&0&0&0 \\ 
0&0&0&0&-1&0 \\ 
0&0&0&0&0&0 \\ 
0&0&0&0&0&0 
\end{bmatrix},  
\]
and for \eqref{nongenericA4} we have 
\[
N_{ {\GL}_4(\CC) \times {\GL}_1(\CC)} = 
\left(
\begin{bmatrix} 
0&0&0&0 \\  
0&0&1&0 \\  
0&0&0&0 \\  
0&0&0&0
\end{bmatrix}, 0 \right) 
\overset{\eqref{indentitygspin6}}{\Longleftrightarrow} 
N_{{\GSO}_6(\CC)} =
\begin{bmatrix} 
0&1&0&0&0&0 \\ 
0&0&0&0&0&0 \\ 
0&0&0&0&0&0 \\ 
0&0&0&0&0&0 \\ 
0&0&0&0&0&-1 \\ 
0&0&0&0&0&0 
\end{bmatrix}. 
\]

Finally, in the case where we have three pairs we are in the situation of \eqref{4gl1}. 
Then we have the following seven non-generic constituents:  
\begin{eqnarray}
Q\left( [\nu^{3/2} \tchi], [\nu^{1/2} \tchi], [\nu^{-1/2} \tchi], [\nu^{-3/2} \tchi] \right) & =  \tchi \circ \det; 
\label{nongenericA5} \\
Q\left( [\nu^{1/2} \tchi, \nu^{3/2} \tchi], [\nu^{-1/2} \tchi], [\nu^{-3/2} \tchi] \right); &
\label{nongenericA6} \\
Q\left( [\nu^{3/2} \tchi], [\nu^{-1/2} \tchi, \nu^{1/2} \tchi], [\nu^{-3/2} \tchi] \right); & 
\label{nongenericA7} \\
Q\left( [\nu^{3/2} \tchi], [\nu^{1/2} \tchi], [\nu^{-3/2} \tchi, \nu^{-1/2} \tchi] \right); & 
\label{nongenericA8} \\
Q\left( [\nu^{1/2} \tchi, \nu^{3/2} \tchi], [\nu^{-3/2} \tchi, \nu^{-1/2} \tchi] \right); & 
\label{nongenericA9} \\ 
Q\left( [\nu^{-1/2} \tchi, \nu^{1/2} \tchi, \nu^{3/2} \tchi], [\nu^{-3/2} \tchi] \right); & 
\label{nongenericA10} \\
Q\left( [\nu^{3/2} \tchi], [\nu^{-3/2} \tchi, \nu^{-1/2} \tchi, \nu^{1/2} \tchi] \right). 
& 
\label{nongenericA11} 
\end{eqnarray}
For \eqref{nongenericA5} we have 
\[ 
N_{ {\GL}_4(\CC) \times {\GL}_1(\CC)} = 
\left( 0_{4\times4},  
0 \right) 
\overset{\eqref{indentitygspin6}}{\Longleftrightarrow} 
N_{{\GSO}_6(\CC)} = 0_{6\times6}, 
\]  
for \eqref{nongenericA6} we have 
\[
N_{ {\GL}_4(\CC) \times {\GL}_1(\CC)} = 
\left(
\begin{bmatrix} 
0&1&0&0 \\  
0&0&0&0 \\  
0&0&0&0 \\  
0&0&0&0
\end{bmatrix}, 0 \right) 
\overset{\eqref{indentitygspin6}}{\Longleftrightarrow} 
N_{{\GSO}_6(\CC)} =
\begin{bmatrix} 
0&0&0&0&0&0 \\ 
0&0&1&0&0&0 \\ 
0&0&0&0&0&0 \\ 
0&0&0&0&-1&0 \\ 
0&0&0&0&0&0 \\ 
0&0&0&0&0&0 
\end{bmatrix}, 
\]
for \eqref{nongenericA7} we have 
\[
N_{ {\GL}_4(\CC) \times {\GL}_1(\CC)} = 
\left(
\begin{bmatrix} 
0&0&0&0 \\  
0&0&1&0 \\  
0&0&0&0 \\  
0&0&0&0
\end{bmatrix}, 0 \right) 
\overset{\eqref{indentitygspin6}}{\Longleftrightarrow} 
N_{{\GSO}_6(\CC)} =
\begin{bmatrix} 
0&1&0&0&0&0 \\ 
0&0&0&0&0&0 \\ 
0&0&0&0&0&0 \\ 
0&0&0&0&0&0 \\ 
0&0&0&0&0&-1 \\ 
0&0&0&0&0&0 
\end{bmatrix}, 
\]
for \eqref{nongenericA8} we have 
\[
N_{ {\GL}_4(\CC) \times {\GL}_1(\CC)} = 
\left(
\begin{bmatrix} 
0&0&0&0 \\  
0&0&0&0 \\  
0&0&0&1 \\  
0&0&0&0
\end{bmatrix}, 0 \right) 
\overset{\eqref{indentitygspin6}}{\Longleftrightarrow} 
N_{{\GSO}_6(\CC)} =
\begin{bmatrix} 
0&0&0&0&0&0 \\ 
0&0&0&1&0&0 \\ 
0&0&0&0&-1&0 \\ 
0&0&0&0&0&0 \\ 
0&0&0&0&0&0 \\ 
0&0&0&0&0&0 
\end{bmatrix}, 
\]
for \eqref{nongenericA9} we have 
\[
N_{ {\GL}_4(\CC) \times {\GL}_1(\CC)} = 
\left(
\begin{bmatrix} 
0&1&0&0 \\  
0&0&0&0 \\  
0&0&0&1 \\  
0&0&0&0
\end{bmatrix}, 0 \right) 
\overset{\eqref{indentitygspin6}}{\Longleftrightarrow} 
N_{{\GSO}_6(\CC)} =
\begin{bmatrix} 
0&0&0&0&0&0 \\ 
0&0&1&1&0&0 \\ 
0&0&0&0&-1&0 \\ 
0&0&0&0&-1&0 \\ 
0&0&0&0&0&0 \\ 
0&0&0&0&0&0 
\end{bmatrix}, 
\]
for \eqref{nongenericA10} we have 
\[
N_{ {\GL}_4(\CC) \times {\GL}_1(\CC)} = 
\left(
\begin{bmatrix} 
0&1&0&0 \\  
0&0&1&0 \\  
0&0&0&0 \\  
0&0&0&0
\end{bmatrix}, 0 \right) 
\overset{\eqref{indentitygspin6}}{\Longleftrightarrow} 
N_{{\GSO}_6(\CC)} =
\begin{bmatrix} 
0&1&0&0&0&0 \\ 
0&0&1&0&0&0 \\ 
0&0&0&0&0&0 \\ 
0&0&0&0&-1&0 \\ 
0&0&0&0&0&-1 \\ 
0&0&0&0&0&0 
\end{bmatrix}, 
\]
and for \eqref{nongenericA11} we have 
\[
N_{ {\GL}_4(\CC) \times {\GL}_1(\CC)} = 
\left(
\begin{bmatrix} 
0&0&0&0 \\  
0&0&1&0 \\  
0&0&0&1 \\  
0&0&0&0
\end{bmatrix}, 0 \right) 
\overset{\eqref{indentitygspin6}}{\Longleftrightarrow} 
N_{{\GSO}_6(\CC)} =
\begin{bmatrix} 
0&1&0&0&0&0 \\ 
0&0&0&1&0&0 \\ 
0&0&0&0&-1&0 \\ 
0&0&0&0&0&0 \\ 
0&0&0&0&0&-1 \\ 
0&0&0&0&0&0 
\end{bmatrix}.  
\]

\item $\bM \s \GL_2 \times \GL_1 \times \GL_1$ and 
$\tbM = ( \GL_2 \times \GL_1 \times \GL_1) \times \GL_1.$ 

Given $\sigma_0 \in \Irr(\GL_2)$ and $\chi_1,\chi_2 \in (F^\times)^D$, we consider 
\begin{equation} \label{ps gspin6 II non}
i_{M}^{\GSpin_6}(\sigma_0 \boxtimes \chi_1 \boxtimes \chi_2).
\end{equation}
Write  
\[
\sigma_0 \boxtimes \chi_1 \boxtimes \chi_2 
= (\tsigma_0 \boxtimes \tchi_1 \boxtimes \tchi_2 \boxtimes \teta)|_{M}
\]
with $\tsigma_0 \in \Irr(\GL_2)$ and $\tchi_i, \teta \in (F^\times)^D$.  
By \eqref{case 2}, it follows that  \eqref{ps gspin6 II non} contains a non-generic constituent 
if and only if its lift 
\begin{equation} \label{ps gl1x4 II non} 
i_{\tM}^{\GL_4 \times GL_1} (\tsigma_0 \boxtimes \tchi_1 \boxtimes \tchi_2 \boxtimes \teta)
\end{equation}
contains a non-generic constituent if and only if 
\begin{equation} \label{ps gl4 II non} 
i_{\GL_2 \times \GL_1 \times \GL_1}^{GL_4} 
(\tsigma_0 \boxtimes \tchi_1 \boxtimes \tchi_2)
\end{equation}
does. 
Recalling $\mathfrak{nongnr}$-(A), it is sufficient to consider the case of 
$\tsigma_0 \in \Irr(\GL_2)$, $\tchi_1 = \nu^{1/2} \tchi$, and $\tchi_2 = \nu^{-1/2} \tchi$ 
for $\tchi \in (F^\times)^D$,  where the segment $\Delta_{\tsigma_0}$ of $\tsigma_0$ 
does not precede either $\tchi_1$ or $\tchi_2$.  
We then have the following sole non-generic constituent:
\begin{equation} \label{nongenericB}
Q([\Delta_{\tsigma_0}], [\nu^{1/2} \tchi], [\nu^{-1/2} \tchi]).
\end{equation}
We have
\[
N_{ {\GL}_4(\CC) \times {\GL}_1(\CC)} = 
\left(
0_{4\times4}, 
0 \right) 
\overset{\eqref{indentitygspin6}}{\Longleftrightarrow} 
N_{{\GSO}_6(\CC)} = 0_{6\times6}. 
\]

\item $\bM \s \GL_3 \times \GL_1$ and $\tbM = (\GL_3 \times \GL_1) \times \GL_1.$

Given a non-generic $\sigma_0 \in \Irr(\GL_3)$ and any $\chi \in (F^\times)^D$, we consider 
\begin{equation} \label{ps gspin6 III non}
i_{M}^{\GSpin_6}(\sigma_0 \boxtimes \chi).
\end{equation}
Write 
\[
\sigma_0 \boxtimes \chi = (\tsigma_0 \boxtimes \tchi \boxtimes \teta)|_{M}
\]
with non-generic $\tsigma_0 \in \Irr(\GL_3)$ and $\tchi, \teta \in (F^\times)^D.$
As in \eqref{case 3} we have
\[
\tsigma_0 =\sigma_0 \tchi, \quad\mbox{ and } \quad \teta= \chi_2 \tchi^{-2}.
\]

As before, \eqref{ps gspin6 III non} contains a non-generic constituent if and only if its lift 
\begin{equation} \label{ps gl1x4 III non} 
i_{\tM}^{\GL_4 \times GL_1} (\tsigma_0 \boxtimes \tchi \boxtimes \teta)
\end{equation}
also contains one if and only if 
\begin{equation} \label{ps gl4 III non} 
i_{\GL_3 \times \GL_1}^{GL_4} (\tsigma_0 \boxtimes \tchi)
\end{equation}
does. 
To have a non-generic $\tsigma_0$ of $\GL_3(F)$, the irreducible representation $\tsigma_0$ 
must be some constituent in a reducible induction. This case has been covered in 
$\mathfrak{nongnr}$-(A) and (B) above.

\item $\bM \s \GL_1 \times \GSpin_4$ and 
$\tbM = (\GL_2 \times \GL_2) \times \GL_1.$  

Given a non-generic $\sigma_0 \in \Irr(\GSpin_4)$, by Section \ref{nongeneric}, 
we know that it must be of the form
\[
{\Res}^{\GL_2 \times \GL_2}_{\GSpin_4} ((\chi \circ \det) \boxtimes \tsigma) 
\]
for $\tsigma \in \Irr(\GL_2).$
For $ \eta \in (F^\times)^D,$ the induced representation  
\begin{equation} \label{ps gspin6 IV non}
i_{M}^{\GSpin_6}((\chi \circ \det) \boxtimes \tsigma \boxtimes \eta)
\end{equation}
contains a non-generic constituent if and only if so does
\[
i_{\GL_2\times \GL_2}^{\GL_4}((\chi \circ \det) \boxtimes \tsigma), 
\]
which is always the case.
Therefore, if $\tsigma$ is supercuspidal, then 
\[
N_{ {\GL}_4(\CC) \times {\GL}_1(\CC)} = 
\left(0_{4\times4},  
0 \right) 
\overset{\eqref{indentitygspin6}}{\Longleftrightarrow} 
N_{{\GSO}_6(\CC)} =
0_{6\times6}.  
\]
If $\tsigma$ is non-supercuspidal, then it suffices to consider the case 
$\tsigma = \St_{\GL_2} \otimes \eta$ with $\eta \in (F^\times)^D$ 
since the other case has been covered in $\mathfrak{nongnr}$-(A). 
Thus, we have
\[
N_{ {\GL}_4(\CC) \times {\GL}_1(\CC)} = 
\left(
\begin{bmatrix} 
0&0&0&0 \\  
0&0&0&0 \\  
0&0&0&1 \\  
0&0&0&0
\end{bmatrix}, 0 \right) 
\overset{\eqref{indentitygspin6}}{\Longleftrightarrow} 
N_{{\GSO}_6(\CC)} =
\begin{bmatrix} 
0&0&0&0&0&0 \\ 
0&0&0&1&0&0 \\ 
0&0&0&0&-1&0 \\ 
0&0&0&0&0&0 \\ 
0&0&0&0&0&0 \\ 
0&0&0&0&0&0 
\end{bmatrix}.  
\]

\item $\bM \s \GSpin_6$ and $\tbM = \GL_4 \times \GL_1.$

Given a non-generic $\sigma \in \Irr(\GSpin_6),$ it must be of the form 
\begin{equation} \label{ps gspin6 E non}
{\Res}^{\GL_4 \times \GL_1}_{\GSpin_6} 
\left( \tchi \circ \det \boxtimes \teta \right) = \chi \circ \det,
\end{equation}
for some $\tchi, \teta \in (F^\times)^D.$
This is the case $Q([\nu^{3/2} \tchi], [\nu^{1/2} \tchi], [\nu^{-1/2} \tchi], [\nu^{-3/2} \tchi])$ in $\mathfrak{nongnr}$-(A).

\end{enumerate}

\subsection{Computation of the Adjoint L-function for $\GSpin_6$} 
We now give explicit expressions for the adjoint $L$-function of each of the representations of $\GSpin_6(F)$.  
Recall that if we have a parameter $(\phi,N)$ with $N$ a nilpotent matrix on the vector space $V$, 
then its adjoint $L$-function is 
\[ L(s, \phi, \Ad) = 
\det \left( 1 - q^{-s} \Ad(\phi)\vert V_N^I \right)^{-1},  
\] 
where $V_N = \ker(N)$, $V^I$ the vectors fixed by the inertia group, and $V_N^I = V^I \cap V_N$. 
Below for the cases where $N$ is non-zero, we write $\ker(\Ad(N))$ and we use $L_\alpha$ to denote 
the root group associated with the root $\alpha$.

We now consider each case. Using \eqref{Ad-L} and Sections \ref{generic section gspin6}, 
and \ref{nongeneric section gspin6}, we have the following. 

\begin{enumerate} 

\item[$\mathfrak{gnr}$-(a)] 
Given $\sigma \in \Irr_{\rm sc}(\GSpin_6)$,  we have 
$\tsigma = \tsigma_0 \boxtimes \teta \in \Irr_{\rm sc}(\GL_4 \times \GL_1).$ 
Then 
\[
L(s, 1_{F^\times}) L(s, \sigma, \Ad)=L(s, \tsigma_0, \Ad_{\widehat{GL}_4}) 
\]
or 
\[ 
L(s, \sigma, \Ad)=L(s, \tsigma_0, \Ad). 
\]

\item[$\mathfrak{gnr}$-(I)] 
Given $\bM \s \GL_1 \times GL_1 \times \GL_1 \times \GL_1$ and 
$\tbM = (\GL_1 \times \GL_1 \times \GL_1 \times \GL_1) \times \GL_1$, we recall 
\[
i_{\GL_1 \times \GL_1 \times \GL_1 \times \GL_1}^{GL_4} (\tchi_1 \boxtimes \tchi_2 \boxtimes \tchi_3 \boxtimes \tchi_4)
\]
must be irreducible. Thus, given $\sigma \in \Irr(\GSpin_6)$ such that 
\[
\sigma = i_{M}^{\GSpin_6}(\tchi_1 \boxtimes \tchi_2 \boxtimes \tchi_3 \boxtimes \tchi_4), 
\]
we have
\[
L(s, \sigma, \Ad)= L(s)^3 \prod_{i \neq j} L(s, \tchi_i\tchi_j^{-1}).
\]

\item[$\mathfrak{gnr}$-(II)] 
Given $\bM \s \GL_2 \times \GL_1 \times \GL_1$ and 
$\tbM = ( \GL_2 \times \GL_1 \times \GL_1) \times \GL_1$,  
for $\sigma_0 \in \Irr_{\rm esq}(\GL_2)$ and $\chi_1,\chi_2 \in (F^\times)^D$,  
we have an irreducible induced representation 
\[
\sigma = i_{M}^{\GSpin_6}(\sigma_0 \boxtimes \chi_1 \boxtimes \chi_2) 
= {\Res}_{\GSpin_6}^{\GL_4 \times \GL_1}
\left( i_{\GL_2 \times \GL_1 \times \GL_1}^{GL_4} (\tsigma_0 \boxtimes \tchi_1 \boxtimes \tchi_2 \boxtimes \teta) \right),
\]
for some $\tsigma_0 \in \Irr_{\rm esq}(\GL_2)$, and $\tchi_i, \teta \in (F^\times)^D$.  
For supercuspidal $\tsigma_0$ we have
\begin{eqnarray*}
L(s, \sigma, \Ad) & = & 
L(s)^2 L(s, \ts_0, \Ad) L(s, \ts_0\times \tchi_1^{-1}) L(s, \ts_0^\vee \times \tchi_1) \\ 
&&  
L(s, \ts_0\times \tchi_2^{-1}) L(s, \ts_0^\vee\times \tchi_2) L(s, \tchi_1\tchi_2^{-1}) L(s, \tchi_2\tchi_1^{-1}).
\end{eqnarray*}
For non-supercuspidal $\tsigma_0 \in \Irr(\GL_2)$, i.e., 
$\sigma_0 = \St_{\GL_2} \otimes \tchi$ for some $\tchi \in (F^{\times})^D$,  we have 
\begin{equation} \label{N12}
\ker \left( \ad \begin{bmatrix} 
0&1&0&0 \\  
0&0&0&0 \\  
0&0&0&0 \\  
0&0&0&0
\end{bmatrix} \right) 
= 
\left\langle 
\begin{bmatrix} 
a&0&0&0 \\  
0&a&0&0 \\  
0&0&b&0 \\  
0&0&0&c
\end{bmatrix}, 
L_{f_1-f_2}, L_{f_1-f_3}, L_{f_1-f_4}, L_{f_3-f_2}, L_{f_3-f_4}, L_{f_4-f_2}, L_{f_4-f_3} 
\right\rangle.
\end{equation} 
It follows that
\begin{eqnarray*}
L(s, \sigma, \Ad)
& = &
L(s)^2 L(s+1) L(s+1, \tchi \tchi_1^{-1}) L(s+1, \tchi \tchi_2^{-1}) \\ 
&& \cdot 
L(s, \tchi^{-1}\tchi_1) L(s, \tchi^{-1}\tchi_2) L(s, \tchi_1\tchi_2^{-1}) L(s, \tchi_2\tchi_1^{-1}).
\end{eqnarray*}

\item[$\mathfrak{gnr}$-(III)] 
Given $\bM \s \GL_3 \times \GL_1$ and $\tbM = (\GL_3 \times \GL_1) \times \GL_1$, 
for $\sigma_0 \in \Irr_{\rm esq}(\GL_3)$ and $\chi \in (F^\times)^D$,  
we have an irreducible induced representation 
\[
\sigma = i_{M}^{\GSpin_6}(\sigma_0 \boxtimes \chi) 
= {\Res}_{\GSpin_6}^{\GL_4 \times \GL_1}
\left( 
i_{\GL_3 \times \GL_1 \times \GL_1}^{GL_4 \times \GL_1} \left( \tsigma_0 \boxtimes \tchi \boxtimes \teta \right)
\right), 
\]
for $\tsigma_0 \in \Irr_{\rm esq}(\GL_3)$ and $\tchi, \teta \in (F^\times)^D$. 
If $\tsigma_0 \in \Irr_ {\rm esq}(\GL_3)$ is supercuspidal, then we have
\[
L(s, \sigma, \Ad) = L(s) L(s, \ts_0, \Ad) L(s, \ts_0\times \tchi^{-1}) L(s, \ts_0^\vee \times \tchi).
\]
For non-supercuspidal $\tsigma_0 \in \Irr_{\rm esq}(\GL_3),$ i.e., 
$\sigma_0 = \St_{\GL_3} \otimes \tchi_0$ for some $\tchi_0 \in (F^\times)^D,$   we have 
\begin{equation} \label{N1223}
\ker \left( \ad \begin{bmatrix} 
0&1&0&0 \\  
0&0&1&0 \\  
0&0&0&0 \\  
0&0&0&0
\end{bmatrix} \right) 
= 
\left\langle \begin{bmatrix} 
a&c&0&0 \\  
0&a&c&0 \\  
0&0&a&0 \\  
0&0&0&b
\end{bmatrix}, 
L_{f_1-f_3}, L_{f_1-f_4}, L_{f_4-f_3} \right\rangle.
\end{equation}
It follows that
\[ 
L(s, \sigma, \Ad) = 
L(s) L(s+1) L(s+2) L(s+1, \tchi \tchi_0^{-1}) L(s+1, \tchi^{-1} \tchi_0).
\]

\item[$\mathfrak{gnr}$-(IV)] 
Given $\bM \s \GL_1 \times \GSpin_4$ and $\tbM = (\GL_2 \times \GL_2) \times \GL_1$, 
we have the representation \eqref{ps gspin6 IV}
\[
\sigma=i_{M}^{\GSpin_6}( \chi \boxtimes \sigma_0)
\]
with $\sigma_0\in \Irr_{\rm esq}(\GSpin_4)$, and $\chi \in (F^\times)^D$.  
We have the irreducible $i_{\GL_2 \times \GL_2}^{GL_4} (\tsigma_1 \boxtimes \tsigma_2)$ 
as in \eqref{ps gl4 IV}, where
$\chi \boxtimes \sigma_0  \subset (\tsigma_1 \boxtimes \tsigma_2 \boxtimes  \teta)|_{M}$ 
with $\tsigma_i \in \Irr_{\rm esq}(\GL_2), \teta \in (F^\times)^D.$ 
Thus, if $\sigma_0$ is supercuspidal (and hence so are $\tsigma_1$ and $\tsigma_2$) we have 
\[ 
L(s, \sigma, \Ad) = L(s) L(s, \ts_1, \Ad)  L(s, \ts_2, \Ad) 
L(s, \ts_1 \times \ts_2^\vee) L(s, \ts_1^\vee \times \ts_1). 
\]

If $\sigma_0$ is non-supercuspidal, with $\tsigma_1$ supercuspidal and $\tsigma_2$ non-supercuspidal, 
i.e., $\tsigma_2=\St_{\GL_2} \otimes \tchi$ for some $\tchi \in (F^{\times})^D$, we have  
\begin{equation} \label{N34}  
\ker \left( 
\ad \begin{bmatrix} 
0&0&0&0 \\  
0&0&0&0 \\  
0&0&0&1 \\  
0&0&0&0
\end{bmatrix} \right)
= 
\left\langle 
\begin{bmatrix} 
a&0&0&0 \\  
0&b&0&0 \\  
0&0&c&0 \\  
0&0&0&c
\end{bmatrix}, 
L_{f_1-f_2}, L_{f_1-f_4}, L_{f_2-f_1}, L_{f_2-f_4}, L_{f_3-f_1}, L_{f_3-f_2}, L_{f_3-f_4} 
\right\rangle, 
\end{equation} 
and it then follows that
\[ 
L(s, \sigma, \Ad) = 
L(s) L(s+1) L(s, \ts_1, \Ad) 
L(s+\frac{1}{2}, \ts_1^\vee \times \tchi) L(s+\frac{1}{2}, \ts_1 \times \tchi^{-1}). 
\]

If $\sigma_0$ is non-supercuspidal, with $\tsigma_1$ non-supercuspidal and $\tsigma_2$ supercuspidal,  
i.e., $\tsigma_1=\St_{\GL_2} \otimes \tchi$ for some $\tchi \in (F^{\times})^D$, then 
$\ker(\ad(N))$ is as in \eqref{N12} and we have 
\[ 
L(s, \sigma, \Ad) = L(s) L(s+1) 
L(s, \ts_2, \Ad) L(s+\frac{1}{2}, \ts_2^\vee \times \tchi) 
L(s+\frac{1}{2}, \ts_2 \times \tchi^{-1}). 
\]

If both $\tsigma_1$ and $\tsigma_2$ are non-supercuspidal, 
i.e., $\tsigma_i = \St_{\GL_2} \otimes \tchi_i$ with 
$\tchi_1, \tchi_2 \in (F^{\times})^D$ satisfying $\tchi_1 \neq \tchi_2  \nu^{\pm1}$, 
we have  
\begin{equation} \label{N1234}
\ker \left(\ad \begin{bmatrix} 
0&1&0&0 \\  
0&0&0&0 \\  
0&0&0&1 \\  
0&0&0&0
\end{bmatrix} \right) 
= 
\left\langle 
\begin{bmatrix} 
a&0&c&0 \\  
0&a&0&c \\  
d&0&b&0 \\  
0&d&0&b
\end{bmatrix}, 
L_{f_1-f_2}, L_{f_1-f_4}, L_{f_3-f_2}, L_{f_3-f_4} 
\right\rangle,
\end{equation}
and it follows that
\[ 
L(s, \sigma, \Ad) = 
L(s) L(s+1)^2 
L(s+1, \tchi_1\tchi_2^{-1}) 
L(s+1, \tchi_1^{-1} \tchi_2) 
L(s, \tchi_1^{-1}\tchi_2) 
L(s, \tchi_1\tchi_2^{-1}).
\]

\item[$\mathfrak{gnr}$-(V)] 
Given $\bM \s \GL_1 \times \GSpin_4$ and $\tbM = (\GL_2 \times \GL_2) \times \GL_1$, 
we consider $\sigma \in \Irr_{\rm esq}(\GSpin_6)$ and 
$\tsigma \in \Irr_{\rm esq}(\GL_4)$ and  $\teta \in (F^\times)^D$ such that 
$\sigma \subset (\tsigma \boxtimes \teta)|_{M}.$  
Then, $\ts$ must be either \eqref{4gl1} or \eqref{2gl2}.

For \eqref{4gl1} (i.e., $\ts=\St_{\GL_4} \otimes \tchi$), we have 
\begin{equation} \label{N122334} 
\ker \left( \ad \begin{bmatrix} 
0&1&0&0 \\  
0&0&1&0 \\  
0&0&0&1 \\  
0&0&0&0
\end{bmatrix} \right) 
= 
\left\langle 
\begin{bmatrix} 
a&b&c&0 \\  
0&a&b&c \\  
0&0&a&b \\  
0&0&0&a
\end{bmatrix}, L_{f_1-f_4} 
\right\rangle,
\end{equation} 
and it follows that
\[
L(s, \sigma, \Ad)= L(s+3)L(s+2)L(s+1). 
\]

For \eqref{2gl2} 
(i.e., $\ttau \in \Irr_{\rm sc}(\GL_2)$), 
we have 
\begin{equation} \label{N1324}  
\ker \left(\ad \begin{bmatrix} 
0&0&1&0 \\  
0&0&0&1 \\  
0&0&0&0 \\  
0&0&0&0
\end{bmatrix} \right) 
= 
\left\langle 
\begin{bmatrix} 
a&c&0&0 \\  
d&b&0&0 \\  
0&0&a&c \\  
0&0&d&b
\end{bmatrix}, L_{f_1-f_3}, L_{f_1-f_4}, L_{f_2-f_3}, L_{f_2-f_4} 
\right\rangle,
\end{equation} 
and it follows that
\[ 
L(s, \sigma, \Ad) = L(s, \ttau, \Ad) L(s, \ttau \times \ttau^\vee).
\]

\item[$\mathfrak{nongnr}$-(A)]

For $Q([\nu^{1/2} \tchi], [\nu^{-1/2} \tchi], [\tchi_3], [\tchi_4]) $ \eqref{nongenericA1}, we have
\begin{eqnarray*} 
L(s, \sigma, \Ad) & = & 
L(s)^3 L(s+1) L(s-1) 
L(s, \tchi_3\tchi_4^{-1}) L(s, \tchi_3^{-1}\tchi_4)   
\\
&& 
\prod_{i=3,4} \left( 
L(s+\frac{1}{2}, \tchi\tchi_i^{-1}) L(s-\frac{1}{2}, \tchi^{-1}\tchi_i) 
L(s-\frac{1}{2}, \tchi\tchi_i^{-1}) L(s+\frac{1}{2}, \tchi^{-1}\tchi_i) \right)
\end{eqnarray*} 
For $Q\left( [\nu \tchi], [ \tchi], [ \nu^{-1} \tchi], [\tchi_4] \right)$ \eqref{nongenericA2}, we have 
\[
L(s, \sigma, \Ad) = 
L(s)^3 
L(s+1)^2 L(s-1)^2 
L(s+2) L(s-2) 
\prod_{t=0,1,-1} \left( L(s+t, \tchi\tchi_4^{-1})L(s+t, \tchi^{-1}\tchi_4) \right),
\]
For $ Q([ \tchi, \nu \tchi], [ \nu^{-1} \tchi], [\tchi_4])$ \eqref{nongenericA3}, we have 
$\ker(\ad(N))$ as in \eqref{N12} and 
\[
L(s, \sigma, \Ad) = 
L(s)^2 L(s-1)^2 L(s-2) 
\prod\limits_{t=-1,0} L(s+t, \tchi \tchi_4^{-1}) 
\prod\limits_{t=\pm1}L(s+t, \tchi^{-1}\tchi_4). 
\]
For $Q\left( [\nu \tchi], [ \tchi, \nu^{-1} \tchi], [\tchi_4] \right)$ \eqref{nongenericA4}, since
\begin{equation} \label{N23} 
\ker \left( \ad \begin{bmatrix} 
0&0&0&0 \\  
0&0&1&0 \\  
0&0&0&0 \\  
0&0&0&0
\end{bmatrix} \right) 
= 
\left\langle 
\begin{bmatrix} 
a&0&0&0 \\  
0&b&0&0 \\  
0&0&b&0 \\  
0&0&0&c
\end{bmatrix}, 
L_{f_1-f_3}, L_{f_1-f_4}, L_{f_2-f_1}, L_{f_2-f_3}, L_{f_2-f_4}, L_{f_4-f_1}, L_{f_4-f_3} 
\right\rangle,
\end{equation}
we have 
\[ 
L(s, \sigma, \Ad) = 
L(s)^2 L(s+2) L(s-1) L(s+1) 
\prod\limits_{t=0, 1} L(s+t, \tchi\tchi_4^{-1}) 
\prod_{t=\pm 1}L(s+t, \tchi^{-1}\tchi_4).  
\]
For $Q([\nu^{3/2} \tchi], [\nu^{1/2} \tchi], [\nu^{-1/2} \tchi], [\nu^{-3/2} \tchi])$ \eqref{nongenericA5}, 
we have 
\[ 
L(s, \sigma, \Ad)= L(s)^3 L(s+1)^3 L(s-1)^3 L(s+2)^2 L(s-2)^2 L(s+3) L(s-3).
\]
For $Q\left( [\nu^{1/2} \tchi, \nu^{3/2} \tchi], [\nu^{-1/2} \tchi], [\nu^{-3/2} \tchi] \right) $ 
\eqref{nongenericA6}, 
we have $\ker(\ad(N))$ is as in \eqref{N12} and 
\[
L(s, \sigma, \Ad)= L(s)^2 L(s-1)^2 L(s+1)^2 L(s-2) L(s+2) L(s-3).
\] 
For $ Q([\nu^{3/2} \tchi], [\nu^{-1/2} \tchi, \nu^{1/2} \tchi], [\nu^{-3/2} \tchi])$ \eqref{nongenericA7}, 
we have $\ker(\ad(N))$ is as in \eqref{N23} and 
\[
L(s, \sigma, \Ad) = L(s)^2 L(s+1)^2L(s+2)L(s-1)^2L(s-3)L(s-2).
\]
For $Q([\nu^{3/2} \tchi], [\nu^{1/2} \tchi], [\nu^{-3/2} \tchi, \nu^{-1/2} \tchi])$ \eqref{nongenericA8}, 
we have $\ker(\ad(N))$ is as in \eqref{N34} and 
\[
L(s, \sigma, \Ad)=  L(s)^2 L(s+1)^2 L(s-1)^2 L(s-2) L(s+2) L(s-3).
\]
For $Q([\nu^{1/2} \tchi, \nu^{3/2} \tchi], [\nu^{-3/2} \tchi, \nu^{-1/2} \tchi])$ \eqref{nongenericA9}, 
we have $\ker(\ad(N))$ is as in \eqref{N1234} and 
\[
L(s, \sigma, \Ad)= L(s) L(s-1)^2 L(s+1) L(s+2) L(s-2) L(s-3).
\]
For $Q([\nu^{-1/2} \tchi, \nu^{1/2} \tchi, \nu^{3/2} \tchi], [\nu^{-3/2} \tchi])$ \eqref{nongenericA10}, 
we have $\ker(\ad(N))$ is as in \eqref{N1223} and 
\[
L(s, \sigma, \Ad)= L(s) L(s-1) L(s-2) L(s+1) L(s-3).
\]
Finally, for $Q\left( [\nu^{3/2} \tchi], [\nu^{-3/2} \tchi, \nu^{-1/2} \tchi, \nu^{1/2} \tchi] \right)$ \eqref{nongenericA11}, 
since
\begin{equation} \label{N2334} 
\ker \left( \ad \begin{bmatrix} 
0&0&0&0 \\  
0&0&1&0 \\  
0&0&0&1 \\  
0&0&0&0
\end{bmatrix} \right) 
= 
\left\langle 
\begin{bmatrix} 
a&0&0&0 \\  
0&b&c&0 \\  
0&0&b&c \\  
0&0&0&b
\end{bmatrix}, 
L_{f_1-f_4}, L_{f_2-f_1}, L_{f_2-f_4} 
\right\rangle,
\end{equation}
we have 
\[ 
L(s, \sigma, \Ad) = 
L(s) L(s+1) L(s-1) L(s-2) L(s-3).  
\]

\item[$\mathfrak{nongnr}$-(B)]
For $Q([\Delta_{\tsigma_0}], [\nu^{1/2} \tchi], [\nu^{-1/2} \tchi])$ \eqref{nongenericB}, 
with say $[\Delta_{\tsigma_0}] = 
i_{\GL_1 \times \GL_1}^{\GL_2}(\teta_1 \boxtimes \teta_2),~\teta_1\teta_2^{-1} \neq \nu^{\pm1}$ 
we have 
\begin{eqnarray*} 
L(s, \sigma, \Ad) 
& = &
L(s)^3 L(s+1) L(s-1)
L(s, \teta_1\teta_2^{-1}) L(s, \teta_1^{-1}\teta_2) 
\\ 
&  &  
\prod_{i=1,2} 
\left( 
L(s-\frac{1}{2}, \teta_i\tchi^{-1}) 
L(s+\frac{1}{2}, \teta_i\tchi^{-1})  
L(s+\frac{1}{2}, \teta_i^{-1}\tchi) 
L(s-\frac{1}{2}, \teta_i^{-1}\tchi) 
\right).  
\end{eqnarray*}

\item[$\mathfrak{nongnr}$-(C)] 
As mentioned before, all the possibilities in this case were covered in (A) and (B) above.

\item[$\mathfrak{nongnr}$-(D)]
For \eqref{ps gspin6 IV non} with $\tsigma$ supercuspidal, we have
\begin{eqnarray*} 
L(s, \sigma, \Ad) 
& = & 
L(s)^2 L(s+1) L(s-1) L(s, \sigma, \Ad) \\ 
& & 
L(s-\frac{1}{2}, \sigma \times \chi^{-1})
L(s+\frac{1}{2}, \sigma \times \chi^{-1})
L(s-\frac{1}{2}, \sigma^\vee \times \chi)
L(s+\frac{1}{2}, \sigma^\vee \times \chi), 
\end{eqnarray*} 
For \eqref{ps gspin6 IV non} with non-supercuspidal 
$\tsigma=\St_{\GL_2} \otimes \eta,~ \eta \in (F^\times)^D$ 
we have 
$\ker(\ad(N))$ as in \eqref{N34} and 
\begin{eqnarray*} 
L(s, \sigma, \Ad) 
& = & 
L(s)^2 L(s+1)^2 L(s-1) 
L(s, \chi\eta^{-1}) 
L(s+1, \chi\eta^{-1}) 
L(s+1, \chi^{-1}\eta) 
L(s, \chi^{-1}\eta). 
\end{eqnarray*} 
Recall that the remaining possibilities in this case were already covered in (A) above.

\item[$\mathfrak{nongnr}$-(E)]
Finally, as mentioned before, all the possibilities in this case we also covered in (A). 

\end{enumerate}

\section{Correction to \cite{acgspin}} \label{sec:corrections}

We take this opportunity to correct the following errors in our earlier work \cite{acgspin}.  They do not affect the main 
results in that paper.
\subsection{Proposition 5.5 and 6.4}
\begin{itemize}
\item Change ``1,2,4,8, if $p=2$" to ``1,2,4,8,..., $2^{[F:\QQ_2]+2},$ if $p=2.$"  
We have $2^{[F:\QQ_p]+2}$ due to the fact that 
$\left| F^{\times} / (F^{\times})^2 \right| 
= 2^{[F:\QQ_2]+2}$. 
\item For  Proposition 5.5, using \cite[Corollary 7.7]{grossprasad92}, it follows that the case of $p=2$ is bounded by $|(\ZZ/2\ZZ)^{4-1}|=8.$ Here $4$ is coming from $\widehat{\GSpin_4}=\GSO(4,\CC).$
\item For  Proposition 6.4, using \cite[Corollary 7.7]{grossprasad92}, it follows that the case of $p=2$ is bounded by $|(\ZZ/2\ZZ)^{6-1}|=32.$ Here $6$ is coming from $\widehat{\GSpin_6}=\GSO(6,\CC).$
\end{itemize}
\subsection{Remark 5.11}
\begin{itemize}
\item 
The formula (5.13) should read as follows: 
\begin{equation}\tag{5.13} 
\Big| \, \Pi_{\vp}\left( \GSpin_4 \right) \Big| = 
\Big| \, \Pi_{\vp}( \GSpin_4^{1,1} ) \Big| = 4, 
\quad \quad 
\left| \, \Pi_{\vp}\left( \GSpin_4^{2,1} \right) \right| = 1.  
\end{equation}
Also, in the following sentence change ``in which case the multiplicity is 2'' to ``in which case the multiplicity 2 
could also occur".  
We thank Hengfei Lu \cite{lu20} for bringing this error to our attention. 

\item In addition, it is more accurate that we use `not irreducible' rather than `reducible' in this Remark 
since one may have indecomposable parameters.  
Alternatively, we may write $\tvp_i|_{W_F}$ is reducible. Thus, at the beginning the Remark, change 
``When $\tvp_i$ is reducible,'' to ``When $\tvp_i$ is not irreducible,''.
\end{itemize}

\begin{table}[p]
\centering
\caption{Representations of $\GSpin_4(F)$}
\label{g4-maintable}
\vspace{-3ex}
$$
\renewcommand{\arraystretch}{1.4}
\begin{array}{|c|l|l|c|}
 \hline&\mbox{${\Res}^{\GL_2 \times \GL_2}_{\GSpin_4}$ of}&\mbox{$L$-packet Structure}
  &\mbox{generic}\\ 
  \hline \hline
\mbox{(a)}& (\ts_1 \boxtimes \ts_2), \quad \ts_2 \s \ts_1\teta, \ts_i \in \Irr_{\rm sc}(\GL_2)
  &\{1\}, \ZZ/2\ZZ, (\ZZ/2\ZZ)^2
  &\bullet 
  \\ 
  \hline
  
\mbox{(b)}& (\ts_1 \boxtimes \ts_2), \quad \ts_2 \not\s  \ts_1\teta, \ts_i \in \Irr_{\rm sc}(\GL_2)
  &\{1\}, \ZZ/2\ZZ 
  &\bullet \\ 
  \hline
  
\mbox{(i)}&   ({\St}_{\GL_2} \boxtimes {\St}_{\GL_2})
= {\St}_{\GSpin_4} \quad\mbox{(irreducible)} 
  &\{1\}  
  &\bullet \\ 
  \hline
  
\mbox{(ii)}&  (i_{\GL_1 \times \GL_1}^{\GL_2}(\chi_{\GL_1 \times \GL_1}^{\GL_2}(\chi_1 \otimes \chi_2) \boxtimes {\St}_{\GL_2} \otimes \chi)\quad
  \mbox{(irreducible)} 
  &\{1\}  
  &\bullet \\ 
  \hline

\mbox{(iii)}&   (i_{\GL_1 \times \GL_1}^{\GL_2}(\chi_1 \otimes \chi_2) \boxtimes i_{\GL_1 \times \GL_1}^{\GL_2}(\chi_3 \otimes \chi_4)), 
  \chi_1 \neq  \nu^{\pm1 \chi_2}, \chi_3 \neq \nu^{\pm1} \chi_4  
  &\{1\}, \ZZ/2\ZZ  
  &\bullet \\ 
  \hline

\mbox{(iv)}&   (\ts \boxtimes {\St}_{\GL_2} \otimes \chi), \quad \ts \in \Irr_{\rm sc}(\GL_2)\quad
  \mbox{(irreducible)} 
  &\{1\}  
  &\bullet \\ 
  \hline

\mbox{(v)}&  (\ts \boxtimes i_{\GL_1 \times \GL_1}^{\GL_2}(\chi_1 \otimes \chi_2)), \quad \ts \in \Irr_{\rm sc}(\GL_2) 
  &\{1\}, \ZZ/2\ZZ
  &\bullet \\ 
  \hline

\mathfrak{nongnr}&   (\chi \circ \det \boxtimes \tsigma), \quad \ts \in \Irr(\GL_2) \quad
  \mbox{(irreducible)} 

  &\{1\} 
  & \\ 
  \hline
    
\end{array}
$$
\end{table}

\begin{table}[p]
\centering
\caption{The adjoint $L$-function $L(s,\sigma,\Ad)$ for $\GSpin_4$}
\label{g4-polestable}
\vspace{-3ex}
$$
\renewcommand{\arraystretch}{1.4}
\begin{array}{|c|l|c|}
 \hline&L(s,\sigma, \Ad)&\operatorname{ord}_{s=1}\\\hline\hline

\mbox{(a)\&(b)} 
&  
L(s, \tsigma_1, \Sym^2 \otimes \omega_{\ts_1}^{-1}) 
L(s, \tsigma_2, \Sym^2 \otimes \omega_{\ts_2}^{-1})
& 0
\\\hline

\mbox{(i)} 
& 
L(s+1)^2
& 0
\\\hline
  
\mbox{(ii)} & 
L(s) L(s+1) L(s, \chi_1\chi_2^{-1})L(s, \chi_1^{-1}\chi_2)
& 0
\\ 
\hline

\mbox{(iii)} 
& 
L(s)^2 
L(s, \chi_1\chi_2^{-1}) 
L(s, \chi_1^{-1}\chi_2) 
L(s, \chi_3\chi_4^{-1}) 
L(s, \chi_3^{-1}\chi_4)
& 0
\\\hline

\mbox{(iv)} 
& 
L(s+1) 
L(s, \tsigma_2, \Sym^2 \otimes \omega_{\ts_2}^{-1})
& 0
\\\hline

\mbox{(v)} & 
L(s)  
L(s, \chi_1\chi_2^{-1}) L(s, \chi_1^{-1}\chi_2) 
L(s, \tsigma_2, \Sym^2 \otimes \omega_{\ts_2}^{-1}) 
& 0
\\ 
\hline

\mathfrak{nongnr}& L(s-1) L(s) L(s+1) L(s, \tsigma, \Ad) 
  & 
1 + \operatorname{ord}_{s=1} L(s, \tsigma, \Ad)
\\ 
\hline
  
\end{array}
$$
\end{table}
\begin{table}[h]
\centering
\caption{Representations of $\GSpin_6(F)$}
\label{maintable}
\vspace{-3ex}
$$
\renewcommand{\arraystretch}{1.5}
\begin{array}{|c|l|c|}
 \hline&\mbox{${\Res}^{\GL_4 \times \GL_1}_{\GSpin_6}$ of}
  &\mbox{generic}\\\hline\hline
\mbox{(a)} 
& 
(\tsigma_0 \boxtimes \teta), \quad \tsigma_0 \in \Irr_{\rm sc}(\GL_4) 
  &\bullet\\\hline

\mbox{(I)}& i_{(\GL_1 \times \GL_1 \times \GL_1 \times \GL_1) \times \GL_1}^{\GL_4 \times GL_1} 
(\tchi_1 \boxtimes \tchi_2 \boxtimes \tchi_3 \boxtimes \tchi_4 \boxtimes \teta) , \quad \tchi_i \neq \nu \tchi_j 
  &\bullet\\\hline

\mbox{(II)}& 
i_{( \GL_2 \times \GL_1 \times \GL_1) \times \GL_1}^{\GL_4 \times GL_1} 
(\tsigma_0 \boxtimes \tchi_1 \boxtimes \tchi_2 \boxtimes \teta) , \quad 
\tsigma_0 \in \Irr_{\rm esq}(\GL_2),  \tchi_1 \neq \nu^{\pm1} \tchi_2 
  &\bullet\\\hline

\mbox{(III)}& i_{(\GL_3 \times \GL_1) \times \GL_1}^{\GL_4 \times GL_1} 
(\tsigma_0 \boxtimes \tchi \boxtimes \teta) , \quad \tsigma_0 \in \Irr_{\rm esq}(\GL_3)
  &\bullet\\\hline  

\mbox{(IV)}& i_{(\GL_2 \times \GL_2) \times \GL_1}^{\GL_4 \times GL_1} 
(\tsigma_1 \boxtimes \tsigma_2 \boxtimes  \teta) , \quad 
\tsigma_i \in \Irr_{\rm esq}(\GL_2), \tsigma_1 \neq \nu^{\pm1} \tsigma_2 
  &\bullet\\\hline

\mbox{(V)}&  (\tsigma \boxtimes \teta) , \quad \tsigma \in \Irr_{\rm esq}(\GL_4) \setminus \Irr_{\rm sc}(\GL_4)
  &\bullet\\\hline

\mbox{(A)}& i_{(\GL_1 \times \GL_1 \times \GL_1 \times \GL_1) \times \GL_1}^{\GL_4 \times GL_1} 
(\tchi_1 \boxtimes \tchi_2 \boxtimes \tchi_3 \boxtimes \tchi_4 \boxtimes \teta) , \quad \tchi_i = \nu \tchi_j 
  & \\\hline

\mbox{(B)}& i_{( \GL_2 \times \GL_1 \times \GL_1) \times \GL_1}^{\GL_4 \times GL_1} 
(\tsigma_0 \boxtimes \tchi_1 \boxtimes \tchi_2 \boxtimes \teta) , \quad 
\tsigma_0 \not\in \Irr_{\rm esq}(\GL_2), \mbox{ or }  \tchi_1 = \nu^{\pm1} \tchi_2 
  & \\ 
  \hline

\mbox{(C)}& i_{(\GL_3 \times \GL_1) \times \GL_1}^{\GL_4 \times GL_1} 
(\tsigma_0 \boxtimes \tchi \boxtimes \teta) , \quad \mbox{non-generic } \tsigma_0  \in \Irr(\GL_3)
  & \\ 
  \hline  

\mbox{(D)}& i_{(\GL_2 \times \GL_2) \times \GL_1}^{\GL_4 \times GL_1} 
((\chi \circ \det) \boxtimes \tsigma \boxtimes  \teta) , \quad \tsigma \in \Irr(\GL_2)
  & \\ 
  \hline

\mbox{(E)}&  (\tchi \circ \det \boxtimes \teta) , \quad 
\tsigma \in \Irr_{\rm esq}(\GL_4) \setminus \Irr_{\rm sc}(\GL_4)
  & \\\hline

\end{array}
$$
\end{table}

\begin{table}[h]

\caption{The adjoint $L$-function $L(s,\sigma,\Ad)$ for $\GSpin_6$}
\label{polestable}
\vspace{-4ex}
$$
\renewcommand{\arraystretch}{1}
\hspace*{-1cm} 
\begin{array}{|c|l|l|c|}
 \hline 
 & 
 \sigma \in \Irr(\GSpin_6(F)) \mbox{ determined by } 
& 
\begin{array}{l}
L(s,\sigma, \Ad) 
\end{array}
& 
\mbox{ord}_{s=1} 
 \\
 \hline\hline

\mbox{(a)} 
& 
\eqref{sc gl4} \, 
\tsigma_0 \in \Irr_{\rm sc}(\GL_4)  
&  
\begin{array}{l}
L(s, \tsigma_0, \Ad)
\end{array}
& 0
\\ 
\hline

\mbox{(I)} 
& 
\eqref{ps gl4} \, 
\tchi_1 \boxtimes \tchi_2 \boxtimes \tchi_3 \boxtimes \tchi_4 \boxtimes \teta
& 
\begin{array}{l}
L(s)^3 \prod_{i \neq j} L(s, \tchi_i\tchi_j^{-1})
\end{array}
& 
0
\\ 
\hline
  
\mbox{(II)} 
& 
\eqref{ps gl4 II} \,  
\tsigma_0 \in \Irr_{\rm sc}(\GL_2) 
&  
\begin{array}{l} 
L(s)^2 L(s, \ts_0, \Ad) 
L(s, \ts_0\times \tchi_1^{-1}) L(s, \ts_0^\vee \times \tchi_1) \\ 
L(s, \ts_0\times \tchi_2^{-1}) L(s, \ts_0^\vee\times \tchi_2) 
L(s, \tchi_1\tchi_2^{-1}) L(s, \tchi_2\tchi_1^{-1})
\end{array}
& 
0
\\ 
\hline

\mbox{(II)} 
& 
\eqref{ps gl4 II} \,  
\tsigma_0 = \St_{\GL_2} \otimes \tchi 
& 
\begin{array}{l}
L(s)^2 L(s+1) L(s+1, \tchi \tchi_1^{-1}) L(s+1, \tchi \tchi_2^{-1}) \\ 
L(s, \tchi^{-1}\tchi_1)L(s, \tchi^{-1}\tchi_2) L(s, \tchi_1\tchi_2^{-1})L(s, \tchi_2\tchi_1^{-1})
\end{array}
& 
0
\\ 
\hline

\mbox{(III)} 
& 
\eqref{ps gl4 III} \, 
\tsigma_0 \in \Irr_{\rm sc}(\GL_3) 
& 
\begin{array}{l}
L(s) L(s, \ts_0, \Ad) L(s, \ts_0\times \tchi^{-1}) L(s, \ts_0^\vee \times \tchi)
\end{array}
& 
0
\\ 
\hline

\mbox{(III)} 
& 
\eqref{ps gl4 III} \, 
\tsigma_0 = \St_{\GL_3} \otimes \tchi_0 
& 
\begin{array}{l} 
L(s) L(s+1) L(s+2) 
L(s+1, \tchi \tchi_0^{-1}) 
L(s+1, \tchi^{-1} \tchi_0)
\end{array}
& 0
\\\hline

\mbox{(IV)} 
& 
\eqref{ps gl4 IV} \, 
\tsigma_i \in \Irr_{\rm sc}(\GL_2) 
& 
\begin{array}{l} 
L(s) L(s, \ts_1, \Ad)  L(s, \ts_2, \Ad) \\ 
L(s, \ts_1 \times \ts_2^\vee)  L(s, \ts_1^\vee \times \ts_1)
\end{array} 
& 
0
\\ 
\hline

\mbox{(IV)} 
& 
\eqref{ps gl4 IV} \, 
\tsigma_1 \in \Irr_{\rm sc}(\GL_2), \tsigma_2 = \St_{\GL_2} \otimes \tchi 
&   
\begin{array}{l} 
L(s) L(s+1) L(s, \ts_1, \Ad) \\ 
L(s+\frac{1}{2}, \ts_1^\vee \times \tchi) L(s+\frac{1}{2}, \ts_1 \times \tchi^{-1})
\end{array} 
& 
0
\\ 
\hline
\mbox{(IV)} 
& 
\eqref{ps gl4 IV} \, 
\tsigma_2 \in \Irr_{\rm sc}(\GL_2), \tsigma_1 = \St_{\GL_2} \otimes \tchi 
& 
\begin{array}{l}
L(s) L(s+1) L(s, \ts_2, \Ad) \\
L(s+\frac{1}{2}, \ts_2^\vee \times \tchi) L(s+\frac{1}{2}, \ts_2 \times \tchi^{-1})
\end{array}
& 
0
\\ 
\hline
\mbox{(IV)} 
&
\eqref{ps gl4 IV} \, 
\tsigma_1= \St_{\GL_2} \otimes \tchi_1 \tsigma_2 = \St_{\GL_2} \otimes \tchi_2 
&  
\begin{array}{l}
L(s) L(s+1)^2 
L(s, \tchi_1^{-1}\tchi_2) L(s, \tchi_1\tchi_2^{-1})
\\
L(s+1, \tchi_1\tchi_2^{-1}) 
L(s+1, \tchi_1^{-1} \tchi_2) 
\end{array} 
& 
0
\\ 
\hline

\mbox{(V)} 
& 
\eqref{4gl1} \, 
\ts=\St_{\GL_4} \otimes \tchi
& 
\begin{array}{l}
L(s+1) L(s+2) L(s+3) 
\end{array}
& 
0
\\ 
\hline
\mbox{(V)} & 
\eqref{2gl2} \, 
\tsigma=\Delta[\nu^{1/2}, \nu^{-1/2}], \ttau \in \Irr_{\rm sc}(\GL_2) 
&  
\begin{array}{l}
L(s, \ttau, \Ad) L(s, \ttau \times \ttau^\vee)
\end{array}
& 
0
\\ 
\hline

\mbox{(A)} 
& 
\eqref{nongenericA1} \,  
Q\left( [ \nu^{1/2} \tchi ], [ \nu^{-1/2} \tchi ], [\tchi_3], [\tchi_4] \right) 
&  
\begin{array}{l} 
L(s-1) 
L(s)^3 
L(s+1) 
L(s, \tchi_3 \tchi_4^{-1}) L(s, \tchi_3^{-1 }\tchi_4)  
\\
\prod\limits_{i=3,4} 
\left( 
\begin{array}{l}
L(s+\frac{1}{2}, \tchi\tchi_i^{-1})  
L(s-\frac{1}{2}, \tchi^{-1}\tchi_i) 
\\ 
L(s-\frac{1}{2}, \tchi\tchi_i^{-1}) 
L(s+\frac{1}{2}, \tchi^{-1}\tchi_i) 
\end{array} 
\right) 
\end{array} 
& 
\geq 1
\\ 
\hline

\mbox{(A)}
&  
\eqref{nongenericA2} \, 
Q\left( [\nu \tchi], [\tchi], [\nu^{-1} \tchi], [\tchi_4] \right) 
&   
\begin{array}{l}
L(s-2) 
L(s-1)^2 
L(s)^3 
L(s+1)^2 
L(s+2) 
\\ 
\prod\limits_{t=-1,0,1} 
\left( 
L(s+t, \tchi\tchi_4^{-1})  L(s+t, \tchi^{-1}\tchi_4) 
\right)
\end{array} 
& 
\geq 2
\\ 
\hline
\mbox{(A)} 
&  
\eqref{nongenericA3} \, 
Q\left( [\tchi, \nu \tchi ], [\nu^{-1} \tchi ], [\tchi_4] \right) 
&   
\begin{array}{l} 
L(s-2) 
L(s-1)^2 
L(s)^2 
\\ 
\prod\limits_{t=-1,0} L(s+t, \tchi \tchi_4^{-1}) 
\prod\limits_{t=-1,1} L(s+t, \tchi^{-1} \tchi_4)
\end{array} 
& 
\geq 2
\\ 
\hline
\mbox{(A)}
&  
\eqref{nongenericA4} \, 
Q\left( [\nu \tchi], [ \tchi, \nu^{-1} \tchi], [\tchi_4] \right)  
&   
\begin{array}{l} 
L(s-1) 
L(s)^2 
L(s+1) 
L(s+2) 
\\ 
\prod\limits_{t=0,1} L(s+t, \tchi \tchi_4^{-1}) 
\prod\limits_{t=-1,1} L(s+t, \tchi^{-1} \tchi_4)
\end{array} 
& 
\geq 1
\\ 
\hline
\mbox{(A)}
&  
\eqref{nongenericA5} \, 
Q\left( [\nu^{3/2}\tchi], [\nu^{1/2}\tchi], [\nu^{-1/2} \tchi ], [\nu^{-3/2} \tchi] \right) 
&   
\begin{array}{l}
L(s-3) L(s-2)^2 L(s-1)^3 L(s)^3 
\\
L(s+1)^3 L(s+2)^2 L(s+3) 
\end{array} 
& 
3
\\ 
\hline

\mbox{(A)}
&   
\eqref{nongenericA6} \, 
Q\left( [\nu^{1/2} \tchi, \nu^{3/2} \tchi], [\nu^{-1/2} \tchi], [\nu^{-3/2} \tchi] \right) 
&    
L(s-3)
L(s-2)  
L(s-1)^2 
L(s)^2 
L(s+1)^2 
L(s+2)  
&  
2
\\ 
\hline

\mbox{(A)}
&    
\eqref{nongenericA7} \, 
Q\left( [\nu^{3/2} \tchi], [\nu^{-1/2} \tchi, \nu^{1/2} \tchi], [\nu^{-3/2} \tchi] \right) 
&    
L(s-3) L(s-2) L(s-1)^2 L(s)^2 L(s+1)^2 L(s+2) 
& 
2
\\ 
\hline

\mbox{(A)}
&   
\eqref{nongenericA8} \, 
Q\left( [\nu^{3/2} \tchi], [\nu^{1/2} \tchi], [\nu^{-3/2} \tchi, \nu^{-1/2} \tchi] \right) 
&    
L(s-3) 
L(s-2) 
L(s-1)^2 
L(s)^2 
L(s+1)^2 
L(s+2) 
& 
2
\\ 
\hline

\mbox{(A)}
&   
\eqref{nongenericA9} \, 
Q\left( [\nu^{1/2} \tchi, \nu^{3/2} \tchi], [\nu^{-3/2} \tchi, \nu^{-1/2} \tchi] \right) 
&    
L(s-3) L(s-2) L(s-1)^2 L(s) L(s+1) L(s+2) 
& 
2
\\ 
\hline

\mbox{(A)}
&   
\eqref{nongenericA10} \, 
Q\left( [\nu^{-1/2} \tchi, \nu^{1/2} \tchi, \nu^{3/2} \tchi], [\nu^{-3/2} \tchi] \right) 
&     
L(s-3) L(s-2) L(s-1) L(s) L(s+1) 
& 
1
\\ 
\hline

\mbox{(A)}
&   
\eqref{nongenericA11} \, 
Q\left( [\nu^{3/2} \tchi ], [\nu^{-3/2} \tchi, \nu^{-1/2} \tchi, \nu^{1/2} \tchi] \right)
&     
L(s-3) L(s-2) L(s-1) L(s) L(s+1) 
& 
1
\\ 
\hline

\mbox{(B)}
&   
\eqref{nongenericB}  
\begin{array}{c}
Q\left( [i_B^{\GL_2}(\teta_1 \boxtimes \teta_2)], [\tchi \nu^{1/2}], [\tchi \nu^{-1/2}] \right)\!, \\ 
\teta_1\teta_2^{-1} \neq \nu^{\pm1}  
\end{array}  
&    
\begin{array}{l}
L(s-1) L(s)^3 L(s+1) 
L(s, \teta_1\teta_2^{-1}) L(s, \teta_1^{-1}\teta_2) \\ 
\prod\limits_{t=\pm \frac{1}{2}}
\prod\limits_{i=1,2} 
\left( 
L(s+t, \teta_i\tchi^{-1}) 
L(s+t, \teta_i^{-1}\tchi) 
\right) 
\end{array} 
& 
\geq 1
\\ 
\hline

\mbox{(B)}
&
\eqref{nongenericB} \, 
\mbox{ (others covered in (A)) } 
&\multicolumn{2}{|c|}{
  }  
\\ 
\hline

\mbox{(C)}
&
\eqref{ps gl4 III non} \, 
\mbox{ (covered in (A) and (B)) } 
&\multicolumn{2}{|c|}{
  }  
\\ 
\hline

\mbox{(D)}
&   
\eqref{ps gspin6 IV non} \, 
\mbox{ with } \tsigma \in \Irr_{\rm sc}(\GL_2) 
&    
\begin{array}{l}
L(s-1) 
L(s)^2 
L(s+1) 
L(s, \sigma, \Ad) 
\\
\prod\limits_{t=\pm\frac{1}{2}}
\left( 
L(s+t, \sigma \times \chi^{-1}) 
L(s+t, \sigma^\vee \times \chi) 
\right) \\ 
\end{array} 
& 
1
\\ 
\hline

\mbox{(D)}
&   
\eqref{ps gspin6 IV non} \, 
\mbox{ with } 
\tsigma = \St_{\GL_2} \otimes \eta  
&    
\begin{array}{l} 
L(s-1) L(s)^2 L(s+1)^2 \\
L(s, \chi\eta^{-1}) L(s+1, \chi\eta^{-1}) L(s+1, \chi^{-1}\eta) L(s, \chi^{-1}\eta)
\end{array} 
& 
\geq 1
\\ 
\hline

\mbox{(D)}
&
\eqref{ps gspin6 IV non} \, 
\mbox{ (others covered in (A)) } 
&\multicolumn{2}{|l|}{
}  
\\ 
\hline

\mbox{(E)}
&
\eqref{ps gspin6 E non} \, 
\mbox{ (covered in (A)) }
& 
\multicolumn{2}{|l|}{
  }  
\\ 
\hline

\end{array} 
$$
\end{table}

\end{document}